\documentclass[10pt,twoside]{siamart1116}

\usepackage[english]{babel}
\usepackage{graphicx,epstopdf,epsfig}
\usepackage{amsfonts,epsfig,fancyhdr,graphics, hyperref,amsmath,amssymb}

\newcommand{\rank}{\mathrm{rank}}

\usepackage{amsfonts}      
\usepackage{amsmath,amscd}
\usepackage{amssymb} 
\usepackage{comment}

\def\endproof{\ifmmode
  \else \leavevmode\unskip\penalty9999 \hbox{}\nobreak\hfill
  \fi
  \quad\hbox{\openbox}}
\newcommand{\openbox}{\leavevmode
  \hbox to.77778em{%
  \hfil\vrule
  \vbox to.675em{\hrule width.6em\vfil\hrule}%
  \vrule\hfil}}
  
  \usepackage{amssymb,amsmath,esint}

\newcommand\R{\mathbb R}
\newcommand\bmat[1]{\begin{bmatrix} #1 \end{bmatrix}}
\DeclareMathOperator\tr{tr}
\DeclareMathOperator\Img{Im}
\DeclareMathOperator\diag{diag}
\DeclareMathOperator\adj{adj}

\newcommand{\Names}{M.I. Bueno, Benjamin Faktor, Rhea Kommerell, Runze Li,  Joey Veltri}
\newcommand{\Title}{Linear maps preserving the Lorentz\\spectrum of  $\mathbf{3 \times 3}$ matrices}

\renewcommand{\theequation}{\thesection.\arabic{equation}}
\renewtheorem{theorem}{Theorem}[section]

\begin{document}

\bibliographystyle{plain}

\setcounter{page}{1}

\thispagestyle{empty}

\title{\Title}
 
\author{
M. I. Bueno\thanks{Department of Mathematics,
 University of California Santa Barbara, Santa Barbara, CA 93106, USA
(mbueno@ucsb.edu). The work of the second and third authors was partially supported by the NSF grant DMS-1850663. This publication is also part of the ``Proyecto de I+D+i PID2019-106362GB-I00 financiado
por MCIN/AEI/10.13039/501100011033''.}
\and
Ben Faktor\thanks{Department of Mathematics,
 University of California Santa Barbara, Santa Barbara, CA 93106, USA
(benjaminfaktor@gmail.com).}
\and
Rhea Kommerell\thanks{Department of Mathematics, University of California Berkeley, Berkeley, CA, USA (rkommerell@berkeley.edu)} 
\and
Runze Li\thanks{Department of Mathematics,
 University of California Santa Barbara, Santa Barbara, CA 93106, USA
(runzeli278@umail.ucsb.edu).}
\and 
Joey Veltri\thanks{Department of Mathematics, The Pennsylvania State University, State College, PA 16801, USA (jveltri@psu.edu)} }

\markboth{\Names}{\Title}

\maketitle

\begin{abstract}
For a given  $3\times 3$ real matrix $A,$ the eigenvalue complementarity problem relative to the Lorentz cone consists of finding a real number $\lambda$ and a nonzero vector $x\in \mathbb{R}^3$ such that $x^T(A-\lambda I)x=0$ and both $x$ and $(A-\lambda I)x$ lie in the Lorentz cone, which is comprised of  all vectors in $\mathbb{R}^3$ forming a $45^\circ$ or smaller angle with the positive $z$-axis. We refer to the set of all solutions $\lambda$ to this  eigenvalue complementarity problem as the Lorentz spectrum of $A.$ Our work concerns the characterization of the linear preservers of the Lorentz spectrum on the space $M_3$ of $3\times 3$ real matrices, that is, the linear maps $\phi:M_3\rightarrow M_3$  such that the Lorentz spectra of $A$ and $\phi(A)$ are the same for all $A.$ We have proven that all such linear preservers take the form $\phi(A)=(Q\oplus [1])A(Q^T\oplus [1]),$ where $Q$ is an orthogonal $2\times 2$ matrix.   \end{abstract}

\begin{keywords}
Lorentz cone, Lorentz  eigenvalues, linear preservers, $3 \times 3$ matrices.
\end{keywords}
\begin{AMS}
15A18, 58C40.
\end{AMS}

\section{Introduction}\label{sec:introduction}

Let $M_n$ denote the vector space of  $n \times n$ real matrices. For a given matrix $A \in M_n$ and a closed convex cone $K \subseteq \mathbb{R}^n,$ the \textit{eigenvalue complementarity problem} consists of finding $\lambda \in \mathbb{R}$ and nonzero $x \in \mathbb{R}^n$ satisfying 
\begin{equation}
    x \in K, \quad (A-\lambda I)x \in K^*, \quad \textrm{and} \quad x^T (A-\lambda I)x = 0,
\end{equation}
where $K^*$ denotes the dual cone of $K$, that is,
$$K^*=\{y\in \mathbb{R}^n:\; x^Ty\geq 0, \; \forall x\in K\}.$$
This problem is a generalization of the standard eigenvalue problem for which $K=\mathbb{R}^n$ and $K^* = \{0\}.$

We are interested in the eigenvalue complementarity  problem on the \textit{Lorentz cone} $\mathcal{K}_n$, given by
$$
\mathcal{K}_n = \left\{ \begin{bmatrix} \xi \\ \eta\end{bmatrix}:\; \xi \in \mathbb{R}^{n-1}, \ \eta \in \mathbb{R}, \ \|\xi\|_2 \leq \eta \right\}.
$$
Notably, the Lorentz cone is self-dual, i.e., $\mathcal{K}_n = (\mathcal{K}_n)^*.$ Thus, the eigenvalue complementarity problem for the Lorentz cone  consists of finding $\lambda\in \mathbb{R}$ and nonzero $x\in \mathbb{R}^n$ satisfying
$$
x \in \mathcal{K}_n, \quad (A-\lambda I)x \in \mathcal{K}_n,\quad \textrm{and} \quad x^T (A-\lambda I)x = 0.
$$
Any such solution $\lambda$ is called a \textit{Lorentz eigenvalue} of $A,$ and any associated $x$ is called a \textit{Lorentz eigenvector}. The collection of all such $\lambda$ is the \textit{Lorentz spectrum} of $A,$ denoted $\sigma_L(A)$. If $\lambda$ has an associated Lorentz eigenvector in the interior (resp.\ boundary) of $\mathcal{K}_n,$ it is called an \textit{interior} (resp.\ \textit{boundary}) \textit{Lorentz eigenvalue}. The collection of all interior (resp.\ boundary) Lorentz eigenvalues is called the \textit{interior} (resp.\ \textit{boundary}) \textit{Lorentz spectrum} of $A$, denoted $\sigma_{int}{(A)}$ (resp.\ $\sigma_{bd}(A)$). Note that $\sigma_L(A)$ is the (not necessarily disjoint) union of $\sigma_{int}(A)$ and $\sigma_{bd}(A).$ One distinctive property of the L-spectrum compared to the standard spectrum of a matrix is that it can be infinite. For the sake of brevity, throughout this paper we will write L-eigenvalue, L-eigenvector and L-spectrum in place of Lorentz eigenvalue, Lorentz eigenvector, and Lorentz spectrum, respectively. 

The
Lorentz cone is an important object of study in several areas of math, especially in optimization. The associated optimization models have applications in several fields such as
 engineering, finance, and  control theory. The Lorentz
cone is also helpful to understand the behavior of some 
linear maps called Z-transformations. Examples of papers that study applications of the Lorentz cone are \cite{second-order, nemeth}.

Recently \cite{maribel1, maribel2,seeger-torki2}, there has been particular interest in the characterization of the linear maps $\phi:M_n \rightarrow M_n$ which preserve the Lorentz spectrum of all matrices, that is, $\sigma_L(A) = \sigma_L(\phi(A))$ for all $ A \in M_n.$ We call such maps  \textit{linear preservers of the Lorentz spectrum}. In studying this problem, we assume that $n\geq 3$. For $n=2$, the Lorentz cone is a polyhedral cone, which is not the case for $n\geq 3.$ The characterization of the linear preservers of the Lorentz spectrum for $n=2$ is an immediate consequence of the characterization of the linear preservers of the Pareto spectrum since the Pareto cone, or nonnegative orthant, is a rotation of the Lorentz cone in $\R^2$ by $45^\circ$ \cite{pareto}.

For $n\geq 3$, some partial results have been proven in the literature. 
\begin{theorem}\label{thm:bijective}\cite{maribel1}
Let $\phi: M_n \rightarrow M_n$ be a linear preserver of the L-spectrum. Then $\phi$ is bijective and $\phi(I_n)= I_n.$
\end{theorem}

\begin{theorem}\label{QQ^TLpres}\cite{maribel1}
Let $Q$ be an orthogonal $(n-1) \times (n-1)$ matrix, and let $\phi:M_n\rightarrow M_n$ be the linear map given by 
\begin{equation}\label{linear-Q}
\phi(A) = \begin{bmatrix} Q & 0 \\ 0 & 1 \end{bmatrix} A \begin{bmatrix} Q^T & 0 \\ 0 & 1 \end{bmatrix}.
\end{equation}
Then $\phi$ is a linear preserver of the L-spectrum. 

Conversely, if $n\geq 3$, $\phi: M_n \to M_n$ is a linear preserver of the L-spectrum, and $\phi(A)= PAQ$ for some fixed $n\times n$ matrices $P$ and $Q$, then $\phi$ satisfies \eqref{linear-Q}. 
\end{theorem}

In \cite{maribel1} it was conjectured that for  $n\geq 3$,  every linear preserver of the L-spectrum must have the form \eqref{linear-Q}. In this paper, we show that this conjecture is true for $n=3$, providing a full characterization of the linear preservers of the L-spectrum in $M_3$.  Moreover, we have shown that the linear preservers of the L-spectrum on $M_3$ also preserve the \textit{nature} of the L-eigenvalues, that is, $\sigma_{int}(A) = \sigma_{int}(\phi(A))$ and  $\sigma_{bd}(A) = \sigma_{bd}(\phi(A))$ for all $A\in M_3.$ The strategy used to prove the characterization of the linear preservers of the L-spectrum on $M_3$ does not seem to be easily generalizable to $n>3$ since it crucially relies on the restrictive form of $3 \times 3$ matrices with infinitely many L-eigenvalues.

The paper is organized as follows: In \cref{sec:background}, we give some properties of the L-spectrum of a matrix as well as a full characterization of the boundary L-eigenvalues. In \cref{infinite-eigenval}, we characterize the $3\times 3$ real matrices with infinitely many L-eigenvalues. In \cref{sec:main-thm}, we present the main results of this paper. The proof of one of these results (\cref{main-thm}) is somewhat cumbersome and is presented in \cref{sec:proof1} and \cref{sec:proof2}.

\section{The Lorentz spectrum of a matrix}\label{sec:background}

In this section, we discuss the Lorentz spectrum of a matrix in more detail. As mentioned in \cref{sec:introduction},  we assume that $n\geq 3$. We also use the notation $\|\cdot\|$ to denote the Euclidean norm of a vector since this is the only vector norm we use in this paper.

Since $\mathcal{K}_n$ is a cone, it is closed under positive scalar multiplication. That is, if $x\in \mathcal{K}_n$ and $\alpha \geq 0$, then $\alpha x \in \mathcal{K}_n$. 
Thus,  if $x=[\tilde{x}^T, x_n]^T\in \mathcal{K}_n$  is an L-eigenvector of a matrix $A$, then $x/x_n$ is also an L-eigenvector of $A$ associated with the same L-eigenvalue since $x_n>0$. Moreover, if $x$ is in the interior (resp.\ boundary) of $\mathcal{K}_n$, so is $x/x_n.$
Thus, from the definition of interior and boundary L-eigenvalues, we have the following results. 
\begin{theorem}\label{thm-interior}
$\lambda\in \mathbb{R}$ is an interior L-eigenvalue of $A \in M_n$ if and only if there exists $ \xi \in \mathbb{R}^{n-1}$ such that $\|\xi\| < 1$ and
$$
(A-\lambda I) \begin{bmatrix} \xi \\ 1\end{bmatrix} = 0.
$$
\end{theorem}

\begin{proof}
Recall that  $\lambda$ is an interior L-eigenvalue of $A$ if and only if there exists an L-eigenvector  $x=[\xi^T, 1]^T$ of $A$ associated with $\lambda$ in the interior of $\mathcal{K}^n$. This is equivalent to the conditions $\|\xi\|<1$,  $(A-\lambda I)x\in \mathcal{K}_n$, and $x^T (A-\lambda I)x=0$. Since $x$ and $(A-\lambda I)x$ must be orthogonal and any pair of nonzero orthogonal vectors in $\mathcal{K}_n$ must lie on the boundary, we have $(A-\lambda I)x =0$, which proves the result.
\end{proof}

Notice that as an immediate consequence of the previous theorem, we have that the interior L-eigenvalues of a matrix $A$ are also standard eigenvalues of $A$ and that the corresponding L-eigenvectors are also standard eigenvectors. 

\begin{theorem}\label{thm:boundary}
$\lambda\in \mathbb{R}$ is a boundary L-eigenvalue of $A \in M_n$ if and only if there exists $ \xi \in \mathbb{R}^{n-1}$ and $s \geq 0$ such that $\|\xi\| = 1$ and
$$
(A-\lambda I) \begin{bmatrix} \xi \\ 1\end{bmatrix} = s \begin{bmatrix} -\xi \\ 1\end{bmatrix}.
$$
\end{theorem}

\begin{proof}
Recall that  $\lambda$ is a boundary L-eigenvalue of $A$ if and only if there exists an L-eigenvector  $x=[\xi^T, 1]^T$ of $A$ associated with $\lambda$ on the boundary of $\mathcal{K}^n$. This is equivalent to the conditions $\|\xi\|=1$,  $(A-\lambda I)x\in \mathcal{K}_n$, and $x^T (A-\lambda I)x=0$. Since 
$(A-\lambda I)x$ must be orthogonal to $x$, $(A-\lambda I)x$ must be a nonnegative multiple of $[-\xi^T, 1]^T$, which proves the result.
\end{proof}

Next we give another characterization of the  boundary  L-eigenvalues of a matrix. We denote the Moore-Penrose inverse of a matrix $M$ by $M^{\dag}$.

\begin{theorem}\label{thm:char-Leig}\cite{seeger-torki}
Let 
\begin{equation}\label{matrixA}
    A=\left[\begin{array}{cc} \tilde{A} & u\\ v^T & a\end{array}\right], \quad \text{where} \quad \tilde{A}\in M_{n-1},\quad u,v\in \mathbb{R}^{n - 1}, \quad \text{and} \quad a\in \mathbb{R}.
\end{equation}
A real number $\lambda$ is in $\sigma_{bd}(A)$ if  one can write $\lambda= \mu+s$, with
$\mu, s\in \mathbb{R}$ and $s\geq 0$, solving (exactly) one of the following systems:

\medskip
{\bf System I}
\begin{enumerate}
\item[I.1] $\mu$ is not an eigenvalue of $\tilde{A}$;
\item[I.2] $v^T (\tilde{A} - \mu I_{n-1})^{-1} u = a - \mu - 2s$;
\item[I.3] $\|(\tilde{A} - \mu I_{n-1})^{-1} u\| = 1.$
\end{enumerate}

\medskip
{\bf System II}
\begin{itemize}
    \item [II.1] $\mu$ is an eigenvalue of $\tilde{A}$;
    \item[II.2] $u\in \Img(\tilde{A} -\mu I_{n-1})$;
    \item[II.3] $v\in \Img(\tilde{A}^T - \mu I_{n-1})$;
    \item[II.4]
    $v^T(\tilde{A}-\mu I_{n-1})^{\dag} u = a - \mu-2s$;
    \item[II.5]
    $$\left\| \left[\begin{array}{cc} \tilde{A}-\mu I_{n-1}\\v^T \end{array} \right]^{\dag}\left[ \begin{array}{c} u \\ a-\mu-2s\end{array} \right]\right \| \leq 1.$$
\end{itemize}

\medskip
{\bf System III}
\begin{itemize}
    \item [III.1] $\mu$ is an eigenvalue of $\tilde{A}$ with geometric multiplicity $1$;
    \item[III.2] $u\in \Img(\tilde{A} -\mu I_{n-1})$;
    \item[III.3] $v\notin \Img(\tilde{A}^T - \mu I_{n-1})$;
    \item[III.4]
    $$\left\| \left[\begin{array}{cc} \tilde{A}-\mu I_{n-1}\\v^T \end{array} \right]^{\dag}\left[ \begin{array}{c} u \\ a-\mu-2s\end{array} \right]\right \| = 1.$$
\end{itemize}

\medskip
{\bf System IV}
\begin{itemize}
    \item [IV.1] $\mu$ is an eigenvalue of $\tilde{A}$ with geometric multiplicity at least $2$;
    \item[IV.2] $u\in \Img(\tilde{A} -\mu I_{n-1})$;
    \item[IV.3] $v\notin \Img(\tilde{A}^T - \mu I_{n-1})$;
    \item[IV.4]
    $$\left\| \left[\begin{array}{cc} \tilde{A}-\mu I_{n-1}\\v^T \end{array} \right]^{\dag}\left[ \begin{array}{c} u \\ a-\mu-2s\end{array} \right]\right \| \leq 1.$$
\end{itemize}

\end{theorem}

A distinguishing property of the L-spectrum compared to the standard spectrum of a matrix is that, while an $n\times n$ real matrix cannot have more than $n$ standard eigenvalues, it may have infinitely many L-eigenvalues. The next theorem characterizes the matrices with this property.

\begin{theorem}\label{thm:infinite}\cite{seeger-torki}
Let $n\geq 3$ and let $A\in M_n$ be partitioned as in \eqref{matrixA}. The  matrix $A$  has infinitely many L-eigenvalues if and only if System IV in \cref{thm:char-Leig} is satisfied for a real eigenvalue $\mu$ of $\tilde A$ and for all $s$ in an interval $[s_1, s_2]$, where $0 \leq s_1 < s_2$.
\end{theorem}

\subsection[Matrices in M3 with infinitely many L-eigenvalues]{Matrices in $M_3$ with infinitely many L-eigenvalues} \label{infinite-eigenval}
Here we characterize the matrices in $M_3$ with infinitely many L-eigenvalues. First we give a technical result, which is used in the proof of \cref{3-inf}. We use the notation $\tr(A)$ for the trace of a matrix $A$.
\begin{lemma}\label{pseudo-rank1}
Let $A$ be a real matrix of rank 1. Then
$$A^{\dag} = \frac{1}{\tr(A^TA)} A^T.$$
\end{lemma}

\begin{proof}
Since $A$ has rank 1, $A = uv^T$ for some nonzero vectors $u$ and $v$. Hence
$$A^TA= vu^T u v^T = \|u\|^2 v v^T$$
also has rank 1.  Thus, since $A^TA$ is symmetric, it is diagonalizable and has exactly one nonzero eigenvalue, namely,  $\tr(A^TA)$.  Hence there exists a nonsingular matrix $P$ such that 
$$A^TA = P \left[ \begin{array}{cc} \tr(A^TA) & 0 \\ 0 & 0 \end{array} \right ] P^{-1}.$$
Let $A = U \Sigma V^T$ be a singular value decomposition of $A$.  Then $$\Sigma = \left[ \begin{array}{cc} \sqrt{\tr(A^TA)} & 0 \\ 0 & 0 \end{array} \right]$$ and
\begin{align*}
    A^{\dag}& = V \Sigma^{\dag} U^T = V  \left[ \begin{array}{cc}  \frac{1}{\sqrt{\tr(A^TA)}} & 0 \\ 0 & 0 \end{array} \right] U^T= \frac{1}{\tr(A^TA)} V \Sigma^T U^T = \frac{1}{\tr(A^TA)} A^T.
    \end{align*}
\end{proof}

The following theorem is the main result in this section.
\begin{theorem}\label{3-inf}
Let $A\in M_3$. Then $A$ has infinitely many L-eigenvalues if and only if
$$A = \left[ \begin{array}{cc} c I_2 & 0 \\ v^T & a \end{array} \right], \quad \text{where}\quad  v \neq 0, \quad a, c\in \mathbb{R},\quad \text{and} \quad c < a +\|v\|.$$
Moreover, 
\begin{equation}\label{L-interval}
\left [\max \left\{c, \frac{a+c -\|v\|}{2} \right\}, \frac{a+c+\|v\|}{2}\right]\subseteq \sigma_{bd}(A).
\end{equation}
\end{theorem}

\begin{proof}
Assume that $A$ is a $3\times 3$ real matrix with infinitely many $L$-eigenvalues. Let us partition $A$ as
$$A=\left[ \begin{array}{cc} \tilde{A} & u \\ v^T & a \end{array} \right], \quad \text{where} \quad \tilde{A}\in M_2,\quad u,v\in \mathbb{R}^2, \quad \text{and} \quad a\in \mathbb{R}.$$
By \cref{thm:infinite}, 
 $A$ must have a boundary L-eigenvalue $\lambda$ satisfying conditions IV.1–IV.4 in \cref{thm:char-Leig}.
Thus, by \cref{thm:boundary},  $\lambda= \mu+s$ with $s\geq 0$, and there is a solution to the following system of equations:
$$\left[ \begin{array}{cc} \tilde{A} - \mu I_2 & u \\ v^T & a - \mu - 2s \end{array} \right]\left[ \begin{array}{c} \xi \\ 1 \end{array} \right] =0, \quad \|\xi\|=1.$$
By condition IV.1, since $\mu$ is an eigenvalue of $\tilde{A}$ of geometric multiplicity 2,
we have $\rank (\tilde{A} - \mu I_2) = 0$. Thus $\tilde{A}= \mu I_2$. Moreover, by condition IV.2, $u\in \Img(\tilde{A}- \mu I_2)$, which means $u=0$. By condition IV.3, we deduce that $v\neq 0$.  Finally, by condition IV.4, we have that
\begin{equation}\label{eq1}
\left \| \left[ \begin{array}{cc} 0 \\ v^T \end{array} \right]^{\dag} \left [ \begin{array}{c} 0 \\ a- \mu - 2s \end{array} \right] \right\| \leq 1.
\end{equation}
By \cref{pseudo-rank1},
$$ \left[ \begin{array}{cc} 0 \\ v^T \end{array} \right]^{\dag} = \frac{1}{\|v\|^2} [0, v].$$
Hence, \eqref{eq1} reduces to
$$\frac{|a - \mu -2s|}{\|v\|} = \frac{1}{\|v\|^2}\left \|[0, v] \left[ \begin{array}{c} 0 \\ a- \mu - 2s \end{array} \right] \right\| \leq 1,$$
or equivalently,
\begin{equation}
\frac{a - \|v\| - \mu}{2} \leq s \leq \frac{a + \|v\|  - \mu}{2}.
\end{equation}
Since there are infinitely many L-eigenvalues and $s\geq 0$, we deduce $a + \|v\| - \mu >0$. Then by taking $c = \mu$, we have an interval of L-eigenvalues $\lambda = \mu + s$ given by \eqref{L-interval}.

The converse holds by \cref{thm:infinite}. 
\end{proof}

\section{Main results}
\label{sec:main-thm}
We now state the two main results of the paper. The first result, together with \cref{QQ^TLpres}, provides a  full characterization of the linear maps $\phi:M_3 \rightarrow M_3$ that preserve the Lorentz spectrum. 
\begin{theorem} \label{main-thm}
Let $\phi:M_3 \rightarrow M_3$ be a linear preserver of the Lorentz spectrum. Then there exists an orthogonal matrix $Q \in M_2$ such that
$$
\phi(A) = \begin{bmatrix} Q & 0 \\ 0 & 1\end{bmatrix} A \begin{bmatrix} Q^T & 0 \\ 0 & 1\end{bmatrix} \quad \textrm{for all $ A \in M_3$}.
$$
We call $Q $ the \emph{orthogonal matrix associated with $\phi$}. 
\end{theorem}

The strategy we employ to prove \cref{main-thm} is to  use  the linearity of the linear preservers applied to a decomposition of  $M_3$ as a direct sum of three subspaces. More explicitly, we decompose $M_3$ as follows:
\begin{equation}\label{3subspaces}
M_3 = \left\{ \begin{bmatrix} 0 & 0 \\ v^T & a \end{bmatrix} \right\} \oplus \left\{ \begin{bmatrix} \tilde{A} & 0 \\ 0 & 0 \end{bmatrix}\right\} \oplus \left\{ \begin{bmatrix} 0 & u \\ 0 & 0\end{bmatrix}\right\}=: \mathcal{S}_1 \oplus \mathcal{S}_2 \oplus \mathcal{S}_3,
\end{equation}
where $\tilde{A} \in M_2,$ $u,v \in \mathbb{R}^2,$ and $a \in \mathbb{R}.$ The image of an arbitrary matrix $A \in M_3$ under $\phi$ is the sum of the images of the projections of $A$ onto each of these  subspaces. 

The proof of \cref{main-thm} is a direct consequence of  \cref{thm:main2}, \cref{thm:main3}, and  \cref{thm:main4}, which give the images of the matrices in $\mathcal{S}_1$, $\mathcal{S}_2$ and $\mathcal{S}_3$, respectively, under a linear preserver of the L-spectrum.

The second main result is presented next and shows that any linear preserver of the L-spectrum on $M_3$ must preserve the nature of the L-eiegenvalues of a matrix. 

\begin{theorem}
Let $\phi:M_3\rightarrow M_3$ be a linear preserver of the L-spectrum. Then for all $A\in M_3$, 
$$\sigma_{int}(A)=\sigma_{int}(\phi(A)) \quad \textrm{and} \quad \sigma_{bd}(A)= \sigma_{bd}(\phi(A)).$$
\end{theorem}

\begin{proof}
Let $\phi$ be a linear preserver of the L-spectrum on $M_3$, and let $A\in M_3$. Since $\phi^{-1}$ is also a linear preserver of the L-spectrum on $M_3$, it is enough to show that $\sigma_{int}(A) \subseteq \sigma_{int}(\phi(A))$ and $\sigma_{bd}(A)\subseteq \sigma_{bd}(\phi(A)).$

Let $\lambda \in \sigma_{int}(A)$. We want to show that $\lambda \in \sigma_{int}(\phi(A)).$ By \cref{thm-interior},  there exists $\xi\in \mathbb{R}^2$ with $\|\xi\|<1$ such that
$$(A-\lambda I_3)\begin{bmatrix} \xi \\ 1 \end{bmatrix} =0.$$
Let $\widehat{Q} = \begin{bmatrix} Q &0 \\ 0 & 1 \end{bmatrix}$, where $Q$ is the orthogonal matrix associated with $\phi$ given by \cref{main-thm}.  Then we have
$$(\phi(A)- \lambda I_3)\begin{bmatrix} Q\xi \\ 1 \end{bmatrix} = \widehat{Q} (A - \lambda I_3) \widehat{Q}^T \widehat{Q} \begin{bmatrix} \xi \\ 1 \end{bmatrix} =0,$$
where $\|Q\xi\| = \|\xi\| <1$ since $Q$ is orthogonal. Thus, $\lambda \in \sigma_{int}(\phi(A))$ and hence $\sigma_{int}(A) \subseteq \sigma_{int}(\phi(A))$. By taking $||\xi|| = 1$ instead of $||\xi|| < 1$, we likewise have $\sigma_{bd}(A) \subseteq \sigma_{bd}(\phi(A))$.
 \end{proof}

\section[Image of matrices in S1 under a linear preserver]{Image of matrices in $\mathcal{S}_1$ under a linear preserver}\label{sec:proof1}
As explained in  \cref{sec:main-thm}, in order to prove \cref{main-thm}, we determine the images of matrices in the three subspaces $\mathcal{S}_1$, $\mathcal{S}_2$ and $\mathcal{S}_3$ given in \eqref{3subspaces} under a linear preserver of the L-spectrum. In this section, we focus on $\mathcal{S}_1.$

We begin with a  lemma that  sheds light on the possible images under a linear preserver of the L-spectrum of matrices in $\mathcal{S}_1$ with infinitely many L-eigenvalues.

\begin{lemma}\label{inf-phi}
Let $\phi : M_3 \rightarrow M_3$ be a linear preserver of the L-spectrum and
$$A =  \left[ \begin{array}{cc} 0 & 0 \\ v^T & a \end{array} \right], \quad \text{where} \quad v \neq 0 \quad \text{and} \quad 0 < a +\|v\|.$$
Then either
\begin{equation}\label{image1}
\phi(A) = \left[ \begin{array}{cc} 0 & 0 \\ w^T & a \end{array} \right], \quad \text{where} \quad \|w\|=\|v\|,
\end{equation}
or 
\begin{equation}\label{image2}
\phi(A) = \left[ \begin{array}{cc} 0 & 0 \\ w^T & a + \|v\| - \|w\| \end{array} \right], \quad \text{where} \quad \frac{a + \|v\|}{2} \leq \|w\| \leq a + \|v\| \quad \text{and} \quad a - \|v\| \leq 0 \leq a.
\end{equation}
\end{lemma}

\begin{proof}
Since $\phi$ preserves the L-spectrum, by \cref{3-inf}, we have
$$\phi(A) =  \left[ \begin{array}{cc} d I_2 & 0 \\ w^T & b \end{array} \right], \quad \text{where} \quad  w \neq 0, \quad d, b\in \mathbb{R},\quad \text{and}\quad d < b +\|w\|.$$
We consider four cases, in which we make repeated use of \cref{spec-inf} to determine the possibilities for $\phi(A)$ and its L-spectrum.

{\bf Case I:} Assume that $0< a - \|v\| $.
In this case,
$$\sigma_L(A)= \{ a\} \cup \left[  \frac{a - \|v\|}{2} , \frac{a + \|v\|}{2} \right] .$$

Hence $\phi(A)$ must have an isolated L-eigenvalue. We have two possible cases for $\phi(A)$:

{\bf Subcase I.1:} $d<b-\|w\| $. In this case, 
$$\sigma_L(\phi(A))= \{b\} \cup \left[  \frac{d+b - \|w\|}{2} , \frac{d+b + \|w\|}{2} \right].$$
Since $\sigma_L(A) = \sigma_L(\phi(A))$, we have
$$b=a, \quad d=0, \quad \text{and} \quad \|v\|= \|w\|,$$
which leads to \eqref{image1}.

{\bf Subcase I.2:} $b-\|w\|<d < b+\|w\|$ and $d> b$. In this case,
$$\sigma_L(\phi(A)) = \{b\} \cup \left [ d, \frac{d+b+\|w\|}{2} \right ].$$
Since $\sigma_L(A) = \sigma_L(\phi(A))$, we have
$$b = a, \quad d = \frac{a-\|v\|}{2}, \quad \text{and} \quad d+\|w\|=\|v\|.$$
However, this implies
$$a = b < d + \|w\| =  \|v\|,$$
a contradiction since $\|v\| < a$ by assumption. So this subcase is impossible.

{\bf Case II:} Assume that $ a - \|v\| =0 $. In this case,
$$\sigma_L(A)= [0, \|v\|].$$
Hence $\phi(A)$ does not have isolated L-eigenvalues. We have two possible cases for $\phi(A)$:

{\bf Subcase II.1:} $d=b-\|w\| $. In this case,
$$\sigma_L(\phi(A))= [d, \|w\|+d].$$
Since $\sigma_L(A) = \sigma_L(\phi(A))$, we have
$$d=0, \quad \|w\|=\|v\|, \quad \text{and} \quad b= d+\|w\|=\|v\|=a,$$
which leads to \eqref{image1}.

{\bf Subcase II.2:} $b-\|w\| < d < b +\|w\|$ and $d \leq b$. 
In this case,
$$\sigma_L(\phi(A))= \left[d, \frac{b+\|w\|+d}{2}\right].$$
Since $\sigma_L(A) = \sigma_L(\phi(A))$, we have
$$ d=0\quad \text{and}  \quad \frac{b+\|w\|}{2}=\|v\|,$$
or equivalently,
$$d = 0\quad \text{and}  \quad b+\|w\|= 2 \|v\| = a +\|v\|.$$
Combining this with the inequalities that define this case, we see that
$$\|w\| > b - d = a + \|v\| - \|w\|  = 2\|v\| - \|w\|$$
and hence $\|w\| > \|v\|$. Similarly,
$$a - \|v\| = 0 = d \leq b = a + \|v\| - \|w\|$$
and hence $\|w\| \leq 2\|v\|$. Altogether, Subcase II.2 gives the conditions
$$d = 0, \quad b = a + \|v\| - \|w\|, \quad \|v\| < \|w\| \leq 2\|v\|, \quad \text{and} \quad a - \|v\| = 0,$$
which leads to \eqref{image2}.

{\bf Case III:} Assume that $a - \|v\| < 0 < a+\|v\|$ and $ a \geq 0$. In this case,
$$\sigma_L(A)= \left [0, \frac{a+\|v\|}{2} \right ].$$

Hence $\phi(A)$ does not have isolated L-eigenvalues.
We have two possible cases for $\phi(A)$:

{\bf Subcase III.1:} $b-\|w\| = d$. In this case,
$$\sigma_L(\phi(A))= [d, \|w\|+d].$$
Since $\sigma_L(A) = \sigma_L(\phi(A))$, we have
$$d=0\quad \text{and}  \quad \|w\|=\frac{a+\|v\|}{2}, \quad b =d+\|w\|=  \frac{a+\|v\|}{2}.$$
Note that $b= a+\|v\|-\|w\|$ and $a-\|v\| \leq 0 \leq a.$ Thus this case leads to \eqref{image2}.

{\bf Subcase III.2:} $b-\|w\|< d < b+\|w\|$ and $d\leq b$. 
In this case,
$$\sigma_L(\phi(A))=\left [d, \frac{b+\|w\|+d}{2}\right ].$$
Since $\sigma_L(A) = \sigma_L(\phi(A))$, we have
$$ d= 0\quad \text{and} \quad b+\|w\|=a+\|v\|. $$

Combining this with the inequalities that define Subcase III.2, we see that
$$\|w\| > b - d = a + \|v\| - \|w\| $$
and hence $\|w\| > \frac{a + \|v\| }{2}$. Similarly,
$$0 = d \leq b = a + \|v\| - \|w\|$$
and hence $\|w\| \leq a + \|v\| $, which leads to \eqref{image2}.

{\bf Case IV:} Assume that $a - \|v\| < 0 < a +\|v\|$ and $ a<0$. In this case,
$$\sigma_L(A)=\{a\} \cup \left [0, \frac{a+\|v\|}{2} \right ].$$
Hence, $\phi(A)$ must have an isolated L-eigenvalue. We have two possible cases for $\phi(A)$:

{\bf Subcase IV.1:} $d < b-\|w\|$. In this case, 
$$\sigma_L(\phi(A))= \{b\} \cup \left[  \frac{d+b - \|w\|}{2} , \frac{d+b + \|w\|}{2} \right].$$
Since $\sigma_L(A) = \sigma_L(\phi(A))$, we have
$$b = a,\quad   \frac{d +a - \|w\|}{2}= 0, \quad \text{and} \quad d +\|w\|=\|v\|.$$
However, this implies
$$a = b > d + \|w\| =  \|v\|,$$
a contradiction since $a - \|v\| < 0$ by assumption. So this subcase is impossible.

{\bf Subcase IV.2:} $b-\|w\|<d< b+\|w\| $ and $d>b$. In this case,
$$\sigma_L(\phi(A)) = \{b\} \cup \left [ d, \frac{d+b+\|w\|}{2} \right ].$$
Since $\sigma_L(A) = \sigma_L(\phi(A))$, we have
$$b = a, \quad d = 0, \quad \text{and} \quad  \|w\|=  \|v\|,$$
which leads to \eqref{image1}.
\end{proof}

We show next that the subspace $\mathcal{S}_1$   is invariant under linear preservers of the L-spectrum.
\begin{lemma}\label{Cphiinvariant}
Let $\phi: M_3 \rightarrow M_3$ be a linear preserver of the L-spectrum. Then the subspace
$$\mathcal{S}_1=\left \{ \left[ \begin{array}{cc} 0 & 0 \\ v^T & a\end{array} \right]: v\in \mathbb{R}^{2}, \; a \in \mathbb{R} \right\}$$
of $M_3$ is $\phi$-invariant, that is, $\phi(\mathcal{S}_1) \subseteq \mathcal{S}_1.$
\end{lemma}

\begin{proof}
We consider three cases:

\medskip{}
Case I: Assume $A=\left[ \begin{array}{cc} 0 & 0 \\ v^T & a\end{array} \right]$ with $v\neq 0$ and $0 < a +\|v\|.$ Then by  \cref{inf-phi}, $\phi(A) \in \mathcal{S}_1.$ 

\medskip
Case II: Assume $A=\left[ \begin{array}{cc}0 & 0 \\ 0 & a\end{array} \right].$
Let $v\in \mathbb{R}^2$ be a nonzero vector such that $0 < a +\|v\|.$ We have
$$A=\left[ \begin{array}{cc}0 & 0 \\ 0 & a\end{array} \right]=\left[ \begin{array}{cc}0 & 0 \\ v^T & a\end{array} \right]-\left[ \begin{array}{cc}0 & 0 \\ v^T & 0\end{array} \right]=: B - M.$$
Note that $B$ and $M$ have infinitely many L-eigenvalues by \cref{3-inf}. Thus, by Case I, $\phi(B), \phi(M)\in \mathcal{S}_1.$
Then since $\phi$ is linear and $\mathcal{S}_1$ is a subspace,
$$\phi(A) =\phi(B)-\phi(M) \in \mathcal{S}_1.$$

\medskip
Case III: Assume $A=\left[ \begin{array}{cc} 0 & 0 \\ v^T & a\end{array} \right]$ with $v\neq 0$ and $0 \geq a+\|v\|$.
Let $d\in \mathbb{R}$ be such that $0 < a +d +\|v\|.$ Then we have
$$A=\left[ \begin{array}{cc}0 & 0 \\ v^T & a\end{array} \right]=\left[ \begin{array}{cc} 0 & 0 \\ v^T & a+d \end{array} \right]-\left[ \begin{array}{cc} 0 & 0 \\ 0 & d\end{array} \right]=: B - M.$$
Note that $\phi(B)\in \mathcal{S}_1$ by Case I and that $\phi(M)\in \mathcal{S}_1 $ by Case II. Then since $\phi$ is linear and $\mathcal{S}_1$ is a subspace,
$$\phi(A) = \phi(B)-\phi(M) \in \mathcal{S}_1.$$
\end{proof}

We show next that we can partition the subspace $\mathcal{S}_1$ into three subsets which are also $\phi$-invariant.

\begin{lemma}\label{lem:inv-sub-S1}
Let $\phi: M_3 \rightarrow M_3$ be a linear preserver of the Lorentz spectrum. Then
$\phi(\mathcal{C}_1) \subseteq \mathcal{C}_1,$ $\phi(\mathcal{C}_2 \cup \mathcal{C}_4) \subseteq \mathcal{C}_2 \cup \mathcal{C}_4$, and $\phi(\mathcal{C}_3 \cup \mathcal{C}_5) \subseteq \mathcal{C}_3 \cup \mathcal{C}_5$, where
$$\mathcal{C}_1:= \left \{ \left[ \begin{array}{cc}0 & 0 \\ v^T & a\end{array} \right]: 0 \neq v\in \mathbb{R}^{2},\; a\in \mathbb{R},\; 0 < a+\|v\| \right\}, $$

$$\mathcal{C}_2:= \left \{ \left[ \begin{array}{cc}  0 & 0 \\ v^T & a\end{array} \right]: 0 \neq v\in \mathbb{R}^{2},\; a\in \mathbb{R},\; a+\|v\|=0 \right\}, $$

$$\mathcal{C}_3:= \left \{ \left[ \begin{array}{cc} 0 & 0 \\ v^T & a\end{array} \right]: 0 \neq v\in \mathbb{R}^{2},\; a\in \mathbb{R},\; a+\|v\|<0 \right\}, $$

$$\mathcal{C}_4:= \left \{ \left[ \begin{array}{cc} 0 & 0 \\ 0 & a\end{array} \right]:  a\in \mathbb{R},\;  a>0\right\}, \quad and $$

$$\mathcal{C}_5:= \left \{ \left[ \begin{array}{cc}0 & 0 \\ 0 & a\end{array} \right]:  a\in \mathbb{R},\;  a \leq 0\right\}. $$
\end{lemma}

\begin{proof}
The result for $\mathcal{C}_1$ follows from \cref{inf-phi}. Then by \cref{spec-inf} and \cref{diagonalca}, every matrix in $\mathcal{C}_2 \cup \mathcal{C}_4$ has exactly two L-eigenvalues, and every matrix in $\mathcal{C}_3 \cup \mathcal{C}_3$ has exactly one L-eigenvalue.
Thus, the results for $\mathcal{C}_2 \cup \mathcal{C}_4$ and for $\mathcal{C}_3\cup \mathcal{C}_5$ follow from  \cref{Cphiinvariant}.
\end{proof}

Next we show that any linear preserver of the L-spectrum restricted to the subspace $\mathcal{C}_4 \cup \mathcal{C}_5$ is the identity map. Henceforth we denote the matrix $e_ie_j^T$ by $E_{ij}$.

\begin{lemma}\label{Enn}
Let $\phi: M_3 \rightarrow M_3$ be a linear preserver of the L-spectrum, and let $A=\left[ \begin{array}{cc} 0 & 0 \\ 0 & a \end{array} \right]$, where $a\in \mathbb{R}$. Then, $\phi(A) =A$.
\end{lemma}

\begin{proof}
Assume first that $A\in \mathcal{C}_4$. Then $a>0$ and $\sigma_L(A)=\{a, \frac{a}{2}\}$ by \cref{diagonalca}.
We know that $\phi(A)\in \mathcal{C}_2\cup \mathcal{C}_4$ by \cref{lem:inv-sub-S1}. If $\phi(A)\in \mathcal{C}_2$, then $\sigma_L(\phi(A))=\{a,0\}$ by \cref{spec-inf}, a contradiction. Therefore, $\phi(A)\in \mathcal{C}_4$, which implies $\phi(A) = A$ by \cref{diagonalca}. In particular, we must have $\phi(E_{33})= E_{33}$. 

Now assume that $A\in \mathcal{C}_5.$ Then $A = aE_{33}$ and $a \leq 0$, so by linearity, we have
$$\phi(A)= \phi(aE_{33}) = a\phi(E_{33})=  a E_{33} = A.$$
\end{proof}

We next present the main result in this section, which gives the image of matrices in the subspace $\mathcal{S}_1$ under linear preservers of the L-spectrum.

\begin{theorem}\label{thm:main2}
Let $\phi:M_3\rightarrow M_3$ be a linear preserver of the L-spectrum. Then there exists an orthogonal $2 \times 2$ matrix $Q$ such that for any matrix $A = \begin{bmatrix} 0 & 0 \\ v^T & a\end{bmatrix}\in \mathcal{S}_1$,
$$\phi(A)=\begin{bmatrix} 0 & 0 \\ (Qv)^T & a\end{bmatrix}.$$
We call $Q$ the \emph{orthogonal matrix associated with $\phi$}.
\end{theorem}

\begin{proof}
By  \cref{Enn} and the linearity of $\phi$, it is enough to show that 
$$
\phi(B):=\phi \left( \begin{bmatrix} 0 & 0 \\ v^T & 0\end{bmatrix}\right) = \begin{bmatrix} 0 & 0 \\ (Qv)^T & 0\end{bmatrix},
$$
for some orthogonal matrix $Q$ independent of $v$. 

This is trivially true for $v = 0$ since $\phi(0)=0$, so assume $v \neq 0$. 
Since $\phi$ preserves the L-spectrum and $B\in \mathcal{C}_1$, by \cref{inf-phi}, we have
$$\phi(B)= \begin{bmatrix}0 & 0 \\ w^T & ||v|| - ||w|| \end{bmatrix}$$
for some $w \in \R^2$ such that $\frac{\|v\|}{2} \leq \|w\| \leq \|v\|$. Let $H=\left[\begin{array}{cc} 0 & 0 \\ v^T & -\|v\|\end{array}\right]\in \mathcal{C}_2$.
By linearity and by \cref{Enn},
\begin{align*}
\phi \left(\left[\begin{array}{cc} 0 & 0 \\ v^T & -\|v\|\end{array}\right]\right)& =\phi\left(\left[\begin{array}{cc} 0 & 0 \\ 0 & -\|v\|\end{array}\right]\right)+\phi \left(\left[\begin{array}{cc} 0 & 0 \\ v^T & 0\end{array}\right]\right)\\
&=\left[\begin{array}{cc} 0 & 0 \\ 0 & -\|v\|\end{array}\right]+\left[\begin{array}{cc} 0 & 0 \\ w^T & \|v\|-\|w\|\end{array}\right]= \begin{bmatrix} 0 & 0 \\ w^T & - \|w\|\end{bmatrix}.
\end{align*}
 Notice that $-\|w\| \in \sigma_{int}(\phi(H))$ and 
 $\sigma_L(H)=\{-\|v\|, 0\}$ by \cref{spec-inf}. Since $||w|| \geq \frac{||v||}{2} > 0$, this implies $\|w\| = \|v\|$. In particular, we have
$$\phi(E_{31}) = \begin{bmatrix} 0 & 0 \\ p^T & 0 \end{bmatrix}\quad \text{and} \quad \phi(E_{32}) = \begin{bmatrix} 0 & 0 \\ q^T & 0 \end{bmatrix}, \quad \text{where} \quad \|p\|=\|q\|=1.$$
Let $p=[p_1, p_2]^T$, $q=[q_1, q_2]^T$, and $Q = \begin{bmatrix} p_1 & q_1 \\ p_2 & q_2\end{bmatrix}.$
Then for any $A=\begin{bmatrix} 0 & 0 \\ v^T & 0 \end{bmatrix}$, where $v = [v_1, v_2]^T$, we have
\begin{align*}
    \phi(A) &= v_1 \phi( E_{31}) + v_2 \phi( E_{32}) = \begin{bmatrix} 0 & 0 \\ v_1p^T +v_2q^T & 0\end{bmatrix} 
    = \begin{bmatrix} 0 & 0 \\ (Qv)^T & 0\end{bmatrix}.
\end{align*}
Since $\|Qv\|= \|v\|$ for all $v$, we deduce that $Q$ is an orthogonal matrix. 
\end{proof}

\section[Image of matrices in S2 and S3 under a linear preserver]{Image of matrices in $\mathcal{S}_2$ and $\mathcal{S}_3$ under a linear preserver}\label{sec:proof2}
We begin this section with some technical lemmas that will be used in the proof of several results. Here we denote the adjugate of a matrix $A$ by $\adj(A)$ and the spectral radius of a matrix $A$ by $\rho(A)$.

\subsection{Auxiliary results}
\begin{lemma}(Rayleigh-Ritz Theorem)
Let $S$ be a symmetric matrix and let $\lambda_{\max}$ be the largest eigenvalue of $S$. Then
$$\lambda_{\max} = \max_{\|x\|=1} x^T S x.$$
\end{lemma}

\begin{lemma}\label{technical2}
Let $B_a=\begin{bmatrix} \tilde{M} & r \\ h^T & a+t\end{bmatrix}$, where $\tilde{M} \in M_{2}$, $h,r\in \mathbb{R}^2,$ and $t\in \mathbb{R}$ are fixed. Assume there is some $a_0 \in \R$ such that $a\in \sigma_L(B_a)$ for all $a \geq a_0$. Then $a$ is an interior L-eigenvalue of $B_a$ for all sufficiently large $a$, and we have
 $$t = 0, \quad h^Tr = 0, \quad \text{and} \quad h^T\adj({\tilde M}) r=0.$$
\end{lemma}

\begin{proof}
Suppose that $a \geq a_0$ is a boundary L-eigenvalue of $\phi(B_a)$. Then, by \cref{thm:boundary}, $a=\mu+s$ for some $s \geq 0$, and there exists $\xi$ with $||\xi|| = 1$ such that
    $$0 = \bmat{(\tilde M -\mu)\xi + r \\ h^T\xi + t - s} = \bmat{\tilde M \xi + (s - a)\xi + r \\ h^T\xi + t - s}.$$
   From the second equation, we have  $s = h^T\xi + t$, and replacing it in the first equation, we get
    $$0 = \tilde M \xi + (s - a)\xi + r = \tilde M \xi + (h^T\xi+ t- a )\xi + r.$$
  Multiplying on the left by $\xi^T$ and taking into account that $\|\xi\|=1$, we have
    \begin{align}
        0 &= \xi^T\tilde M \xi + h^T\xi +t- a  + \xi^Tr \\
        &= \frac{1}{2} \xi^T \tilde{M} \xi + \frac{1}{2} \xi^T \tilde{M} \xi + h^T \xi + r^T \xi + t-a \nonumber \\
        &= \frac{1}{2}\xi^T(\tilde M + \tilde M^T)\xi + (h+r)^T\xi+t - a \nonumber \\
        &\leq \frac{1}{2}\lambda_{\max}(\tilde M + \tilde M^T) + || h+r||+t - a,\label{ainequality1}
    \end{align}
    where the third equality follows from the fact that $\xi^T \tilde{M} \xi$ is a number and, consequently, is equal to its transpose. The inequality follows from the Rayleigh-Ritz theorem and  the Cauchy-Schwarz inequality. 
    
    Hence for $a > \frac{1}{2}\lambda_{\max}(\tilde M + \tilde M^T) + ||h+r|| + t$, condition \eqref{ainequality1}  fails, which means that $a$ must be an interior L-eigenvalue of $\phi(B_a)$ for all sufficiently large $a$. This implies there is some $\xi$ with $||\xi|| < 1$ such that
    \begin{align*}
        0 & = \bmat{\tilde M - aI & r \\ h^T & t}\bmat{\xi \\ 1} = \bmat{(\tilde M - aI)\xi + r \\ h^T\xi + t}.
    \end{align*}
    For $a > \rho (\tilde M)$, we know that $\tilde M - aI$ must be invertible. From the first equation, we have $\xi = -(\tilde M - aI)^{-1}r$, and replacing it in the second equation, we get
    \begin{align*}
        0 &= h^T\xi + t = t - h^T(\tilde M - aI)^{-1}r = t - \frac{h^T \adj(\tilde{M}-aI) r}{\det(\tilde{M}-aI)}
        = t - \frac{h^T(\adj({\tilde{M}}) - aI)r}{a^2 - a\tr(\tilde M) + \det(\tilde M)}.
    \end{align*}
   This implies
    \begin{align*}
        0 &= t(a^2 - a\tr(\tilde M) + \det(\tilde M)) - h^T\adj({ \tilde M}) r+ ah^Tr \\&= ta^2 + (h^Tr - t \tr(\tilde M))a + t\det(\tilde M) - h^T \adj({\tilde M}) r.
    \end{align*}
    Since this holds for all sufficiently large $a$, we must have $t = 0$, $h^Tr = h^T r - t \tr (\tilde{B}) = 0$, and $h^T\adj({\tilde M}) r = - t \det (\tilde{M})+h^T \adj({\tilde M}) r = 0$.  
    \end{proof}
    
    \subsection[Image of matrices in S2 under a linear preserver]{Image of matrices in $\mathcal{S}_2$ under a linear preserver}
Even though our focus on this section is matrices in $\mathcal{S}_2$, we start with a partial result for matrices in $\mathcal{S}_3$. 

\begin{lemma}\label{lem:0u-first}
    Let $\phi: M_3 \to M_3$ be a linear preserver of the L-spectrum with associated orthogonal matrix $Q$, and let
    $$B=\begin{bmatrix}0 & u \\ 0 & 0 \end{bmatrix}, \quad \text{where} \quad u\neq 0.$$
    Then $$\phi(B) = \bmat{\tilde B & Qu \\ w^T & 0}, \quad \text{where} \quad\det(\tilde{B})=0,  \quad w^T Qu =0, \quad
 \textrm{and} \quad \adj(\tilde{B}) Qu = \tr(\tilde{B}) Qu.$$
    \end{lemma}

    \begin{proof}
    Let $\phi(B) = \begin{bmatrix} \tilde{B} & v \\ w^T & b \end{bmatrix}$, $a > ||u||$, and $B^{\perp}=\left[\begin{array}{cc} 0 & 0 \\z^T & a \end{array}\right],$ where $z$ is any vector orthogonal to $u$. 
    By  \cref{thm:main2}, $\phi(B^{\perp}) = \left[\begin{array}{cc}0 & 0 \\ (Qz)^T& a\end{array} \right]$.
    Thus,
    $$\phi(B+ B^{\perp}) = \phi(B) + \phi(B^{\perp}) = \bmat{\tilde B & v \\ (w+Qz)^T & a + b}.$$
    Since $a$ is an (interior) L-eigenvalue of $B+ B^{\perp}$ with L-eigenvector $[u^T/a,1]^T$, we know $a\in \sigma_L(\phi(B+B^{\perp})$. 
    By  \cref{technical2},  we have  
    $b = 0$, $ (w+Qz)^Tv =  0$, and $(w+Qz)^T\adj({\tilde B}) v = 0$.
    By setting $z=0$, we have, in particular, $w^Tv=0$ and $w^T\adj({\tilde B})v=0$, which also implies $z^TQ^Tv=z^TQ^T\adj (\tilde{B})v=0$ for all $z$ orthogonal to $u$.    

 Let $c\in \mathbb{R}$, $c>0$, and consider now the matrix $H:=B+B^T+cE_{33}$. Then, by  \cref{thm:main2}, 
$$\phi(H)=\bmat{\tilde{B} & v \\ (w+Qu)^T & c}.$$
By  \cref{symmetricua},  $m:=\frac{c + \sqrt{c^2 + 4||u||^2}}{2}$ is an (interior) L-eigenvalue of $H$, so $m$ is also an L-eigenvalue of $\phi(H).$

\bigskip

{\bf Step 1:} Suppose $m$ is a boundary L-eigenvalue of $\phi(H)$. Then, using an argument similar to that used in the first part of the  proof of  \cref{technical2}, we get
    \begin{align*}
        0 &= \xi^T\tilde B\xi + (w+Qu)^T\xi +c-2m + \xi^Tv \\
        &\leq \frac{1}{2}\lambda_{\max}(\tilde B + \tilde B^T) + ||w+ Qu + v|| - \sqrt{c^2 + 4||u||^2}.
    \end{align*}
    Hence for $c > \sqrt{\left(\frac{1}{2}\lambda_{\max}(\tilde B + \tilde B^T) + ||w+ Qu + v||\right)^2 - 4||u||^2}$ (or for any $c$ if the radicand is negative), this condition fails, which means that for sufficiently large $c$, $m$ must be an interior L-eigenvalue of $\phi(H)$.
    
    \bigskip

    {\bf Step 2:} Because $m$ is an interior L-eigenvalue of $\phi(H)$ for sufficiently large values of $c$, there must be some $\xi$ with $||\xi|| < 1$ such that
    \begin{align*}
        0 &=  \bmat{\tilde B - mI)\xi + v \\ (w+Qu)^T\xi + c-m}.
    \end{align*}
    Since $m$ approaches $\infty$ as $c$ increases, we know that $m$ will exceed the spectral radius of $\tilde B$ for all sufficiently large $c$. Thus, from the first equation, we get 
    $$\xi = -\left(\tilde B - mI\right)^{-1}v.$$
   Replacing $\xi$ in the second equation, 
    \begin{align*}
        0 &= (w+Qu)^T\xi + c-m=c -m -(w+Qu)^T(\tilde{B}-mI)^{-1} v\\
        &= c-m -\frac{(w+Qu)^T(\adj({\tilde B})-mI)v}{m^2-m\tr(\tilde{B}) + \det(\tilde{B})} 
    \end{align*}
    Then, taking into account that $(c-m)m=-\|u\|^2$, denoting $x=w+Qu$, we have
    \begin{align*}
       0&=(c-m)[m^2-m\tr(\tilde{B})+\det(\tilde{B})] - x^T(\adj({\tilde B})-mI)v \\
        &=-\|u\|^2 m +\|u\|^2 \tr(\tilde{B})+(c-m)\det(\tilde{B}) - x^T\adj({\tilde B}) v + mx^Tv \\
        &= [\|u\|^2 \tr(\tilde{B})-x^T\adj({\tilde B}) v]+[x^Tv - \|u\|^2-\det(\tilde{B})]m+c \det({\tilde{B}}) \\
        &= [\|u\|^2 \tr(\tilde{B})-x^T\adj({\tilde B})v]+[x^Tv - \|u\|^2-\det(\tilde{B})]\left(m-\frac{c}{2}\right)+[x^Tv - \|u\|^2+\det(\tilde{B})]\frac{c}{2}.
    \end{align*}
    Therefore,
    $$\left([x^Tv-\|u\|^2-\det(\tilde{B})]\left(m-\frac{c}{2}\right)\right)^2=\left([\|u\|^2 \tr(\tilde{B})-x^T\adj (\tilde{B})v]+[x^Tv - \|u\|^2+\det(\tilde{B})]\frac{c}{2}\right)^2$$
    or equivalently,
    $$\left(x^Tv-\|u\|^2-\det(\tilde{B})\right)^2\frac{c^2+4\|u\|^2}{4}=\left([\|u\|^2 \tr(\tilde{B})-x^T\adj (\tilde{B})v]+[x^Tv - \|u\|^2+\det(\tilde{B})]\frac{c}{2}\right)^2.$$

    Grouping terms to rewrite this expression as a polynomial in $c$ and using the identity $a^2-b^2=(a+b)(a-b)$ for all $a,b \in \mathbb{R}$, we get
    \begin{align*}
        0= \ &(x^Tv-\|u\|^2)(-\det(\tilde{B})) c^2 - [\|u\|^2\tr(\tilde{B})-x^T\adj (\tilde{B})v][x^Tv-\|u\|^2+\det(\tilde{B})]c\\
        &+[x^Tv-\|u\|^2 + \det(\tilde{B})]\|u\|^2-[\|u\|^2 \tr(\tilde{B})-x^T\adj (\tilde{B})v]^2  
    \end{align*}
   Since $c$ can take arbitrarily large positive values, we deduce
   $$\det(\tilde{B})=0, \quad x^Tv=\|u\|^2, \quad \textrm{and} \quad x^T\adj (\tilde{B})v=\|u\|^2\tr(\tilde{B}).$$
    Recall that we showed above that   $w^Tv = w^T\adj(\tilde{B}) v = 0$. Since $x=w+Qu$, we get
    $$\|u\|^2=x^Tv=u^TQ^Tv\quad \textrm{and} \quad  
    ||u||^2\tr(\tilde B)= x^T\adj (\tilde{B}) v =u^TQ^T\adj (\tilde{B})v.$$ Recall also  that $z^TQ^Tv = z^TQ^T\adj (\tilde{B})v = 0$ for any $z$ orthogonal to $u$. This gives us the following relation:
    \begin{align*}
        \bmat{u^TQ^T \\ z^TQ^T}\bmat{v & \adj(\tilde{B})v} &= ||u||^2\bmat{1 & \tr\tilde B \\ 0 & 0},
    \end{align*}
    where the matrix on the left is nonsingular for $z \neq 0$. Hence, for each $z$, the equation has a unique solution $\bmat{v & \adj (\tilde{B})v}$. By inspection, we note that $v = Qu$ and $\adj(\tilde{B})v = \tr(\tilde B)Qu$ satisfy it, so these must be the true vectors. We then obtain  $\adj (\tilde{B}) Qu=\tr(\tilde{B})Qu$, and the result follows.
\end{proof}

Now we start analyzing the behavior of the matrices in $\mathcal{S}_2$ under the linear preservers of the L-spectrum. As a byproduct of this work, we obtain further information about the images of matrices in $\mathcal{S}_3$ under such linear preservers.

\begin{lemma}\label{[A 0; 0 0] partial result}
Let $\phi: M_3\rightarrow M_3$ be a linear preserver of the L-spectrum with associated orthogonal matrix  $Q$, and let $\tilde A \in M_2$ and $u \in \R^2$. Then for some $\tilde C, \tilde B \in M_2$,
$$\phi\left(\bmat{\tilde A & 0 \\ 0 & 0}\right)=\left[\begin{array}{cc} \tilde{C} & 0 \\0 & 0 \end{array} \right]\quad \textrm{and} \quad \phi\left(\bmat{0 & u \\ 0 & 0}\right) = \bmat{\tilde B & Qu \\ 0 & 0}.
$$

\end{lemma}

\begin{proof}
Let $A=\bmat{ \tilde{A} & 0 \\0 & 0}$ and $\phi(A) = \begin{bmatrix} \tilde{C} & p \\ q^T & c \end{bmatrix}$. 
Let  $a > 0$, $z\in \mathbb{R}^2$ be any nonzero vector, and 
$$H:=\left[\begin{array}{cc} 0 & 0 \\ z^T & a \end{array} \right].$$
By  \cref{thm:main2}, we have $\phi(H) = \left[\begin{array}{cc} 0 & 0 \\ (Qz)^T & a \end{array} \right]$. Thus,
    $$\phi(A+ H ) = \phi(A) + \phi(H) = \bmat{\tilde C & p \\ (q+Qz)^T & c + a}.$$
Notice that $a$ is an (interior) L-eigenvalue of $A + H$ with L-eigenvector $[0, 1]^T$. Since $\phi$ preserves the L-spectrum, $a\in \sigma_L(\phi(A+H))$.  By  \cref{technical2},  we have $c=0$ and  $(q+Qz)^T p=0$. Since $Q$ is invertible and $z\neq 0$ is arbitrary, we have $(q+w)^T p=0$ for all $w\neq 0$, which implies $p=0.$ Thus,
$$\phi\left(\left[\begin{array}{cc} \tilde A & 0 \\0 & 0 \end{array} \right]\right)=\left[\begin{array}{cc} \tilde{C} & 0 \\q^T & 0 \end{array} \right].$$

\medskip
Now we show that $q=0$. For any arbitrary nonzero vector $u\in \mathbb{R}^2$, let $a\in \mathbb{R}$ be large enough so that, by \cref{lem:a-int}, $a$ is an interior L-eigenvalue of
$$G:=\left[\begin{array}{cc} \tilde A & u \\0 & a \end{array} \right].$$
By the linearity of $\phi$ and by   \cref{lem:0u-first} and \cref{Enn}, we have
$$\phi(G)=\left[\begin{array}{cc} \tilde{B}+\tilde{C} & Qu \\(w+q)^T & a \end{array} \right], \quad \text{where} \quad w^TQu=0\quad \text{and} \quad \adj({\tilde B}) Qu= \tr (\tilde{B})Qu.$$
We know that for large enough $a$, $a\in \sigma_L(G)$ and hence $a\in \sigma_L(\phi(G))$. By  \cref{technical2}, $a$ is an interior L-eigenvalue of $\phi(G)$ for large enough $a$, $(w+q)^TQu=0$, and $(w+q)^T \adj(\tilde B + \tilde C) Qu=0$. Since $w^TQu=0$, we deduce $q^T Qu=0$. Since $u$ is arbitrary and independent of $q$, we deduce that $q=0.$ Thus, the  claim for matrices in $\mathcal{S}_2$ follows.

Now observe that
\begin{align*}
    0 &= w^T\adj(\tilde B + \tilde C)Qu = w^T\adj(\tilde B)Qu + w^T\adj(\tilde C)Qu 
    =\tr(\tilde B) w^TQu + w^T\adj({\tilde C})Qu  = w^T\adj(\tilde C)Qu.
\end{align*}

Note that we have shown that matrices $\tilde A \oplus [0]$ map to $\tilde C \oplus [0]$. Since $\phi$ is bijective by \cref{thm:bijective}, any matrix $\tilde C \oplus [0]$ must have a preimage $\tilde A \oplus [0]$. In particular, we may take $\tilde C = \adj(R)$, where $R = \bmat{0 & -1 \\ 1 & 0}$. Hence, $w^TRQu = 0$. Since $w^T Qu=0$ and since the vectors $Qu$ and $RQu$ are linearly independent, we know furthermore that $w = 0$ for all nonzero $u$, which proves the claim for matrices in $\mathcal{S}_3.$
\end{proof}

Next we show that the matrix $\tilde{C}$ in \cref{[A 0; 0 0] partial result} is closely related to $\tilde{A}$.

\begin{lemma}\label{DAD or DA^TD}
Let $\phi: M_3\rightarrow M_3$ be a linear preserver of the L-spectrum with associated orthogonal matrix  $Q$. Then there exists an invertible diagonal matrix $D$ such that either
\begin{equation}\label{phidiagsin}
\phi\left(\left[\begin{array}{cc} \tilde A & 0 \\0 & 0 \end{array} \right]\right) =  \left[\begin{array}{cc} QD \tilde A D^{-1} Q^T & 0 \\0 & 0 \end{array}\right] \quad \text{for all } \tilde A \in M_2,
\end{equation}
or
\begin{equation}\label{phidiagcon}
\phi\left(\left[\begin{array}{cc} \tilde A & 0 \\0 & 0 \end{array} \right]\right) =  \left[\begin{array}{cc}QD \tilde A ^T D^{-1}Q^T & 0 \\0 & 0 \end{array}\right] \quad \text{for all } \tilde A \in M_2.
\end{equation}
\end{lemma}

\begin{proof}
Let 
$W_3:=\{\tilde A \oplus [a]: \tilde A\in M_2, \; a\in \mathbb{R}\}$.   
By  \cref{Enn} and  \cref{[A 0; 0 0] partial result},
the linear map $\tilde\phi: W_3 \to W_3$ given by $\tilde\phi(A) = \phi(A)$ preserves the Lorentz spectrum on $W_3$. Thus, by Theorem 4.2 in \cite{maribel1} with $\mathcal{M}=W_3$, there exists some invertible matrix $P \in M_2$ such that either
\begin{equation}\label{phitype1}
\phi\left(\bmat{\tilde A & 0 \\ 0 & 0}\right) = \tilde\phi\left(\bmat{\tilde A & 0 \\ 0 & 0}\right) = \bmat{P\tilde AP^{-1} & 0 \\ 0 & 0} \quad \textrm{for all $\tilde A \in M_2$,}
\end{equation}
 or
\begin{equation}\label{phitype2}
\phi\left(\bmat{\tilde A & 0 \\ 0 & 0}\right) = \tilde\phi\left(\bmat{\tilde A & 0 \\ 0 & 0}\right) = \bmat{P\tilde A^TP^{-1} & 0 \\ 0 & 0}\quad \textrm{for all $\tilde A \in M_2$.}
\end{equation}
 We will now show that $P = QD$ for some invertible diagonal matrix $D$. Consider matrices of the form
\begin{equation}\label{lem20-case3}
    B=\bmat{D & 0 \\ v^T & a}= \left[\begin{array}{ccc} d_1 & 0 & 0 \\ 0 & d_2 &0 \\ v_1 & v_2 & a\end{array} \right]\in M_3,
\end{equation}
where  $d_1 \neq d_2$, and $v_1+a-d_1 \geq 0.$  Then
$\lambda= \frac{v_1+a+d_1}{2} \in \sigma_{bd}(B)$ with $\mu=d_1$, $s=\frac{v_1+a-d_1}{2}$, and $\xi =[1, 0]^T.$
Let us additionally assume that  $\lambda \notin \sigma(B) = \{a, d_1, d_2\}$.
Then $\lambda\notin \sigma(\phi(B))$, and by \eqref{phitype1} and \eqref{phitype2}, 
$$\phi(B)=\left[\begin{array}{cc}P D P^{-1}& 0 \\ (Qv)^T& a\end{array} \right].$$
As $\sigma_L(B)=\sigma_L(\phi(B))$, we deduce that $\lambda\in \sigma_{bd}(\phi(B))$.  Thus, $\lambda = \mu+s$ with $s\geq 0$, and there is some $\xi$ such that $\|\xi\|=1$ and
$$(PDP^{-1}-\mu I_2) \xi =0,$$
$$(Qv)^T \xi + a - \mu-2s =0.$$
From the first equation, we get
$$(D-\mu I_2) P^{-1}\xi =0.$$
Let $x= P^{-1} \xi,$ that is, $\xi=Px = x_1 p_1 + x_2 p_2$, where $p_i$ denotes the $i$th column of $P$. We consider two cases:

{\bf Case 1:} $\mu=d_1$. Then since $d_1\neq d_2$,  $x_2=0$ and $\xi=x_1 p_1$. As $\|\xi\|=1$, we have $x_1=\pm \frac{1}{\|p_1\|}.$ Thus,
\begin{align*}
    0&=(Qv)^T \xi + a - \mu - 2s = (Qv)^T x_1p_1+ a- \mu -2(\lambda-\mu) \\
    &= (Qv)^T x_1p_1 + a+d_1 -(v_1+a+d_1) 
    =  v^T(x_1 Q^T p_1) -v_1.
\end{align*}
Hence,
$$v^T(x_1Q^Tp_1)= v_1= v^T \left[\begin{array}{c} 1 \\0 \end{array} \right]=v^Te_1.$$
Since this equality holds for any $v$ with $v_1 + a - d_1 \geq 0$ and $\frac{v_1+a+d_1}{2} \notin \{a, d_1, d_2\}$, we deduce that
$x_1 Q^T p_1 = e_1,$
or equivalently,
$$p_1 = \frac{1}{x_1}Qe_1 = \pm\|p_1\|q_1,$$
where $q_1$ denotes the first column of $Q$.
By multiplying $P$ and $P^{-1}$ by $\pm 1$, we may assume without loss of generality that $p_1 = \|p_1\|q_1$.

{\bf Case 2:} $\mu \neq d_1$. Then since $\xi\neq 0$, $x_1=0$ and $\mu=d_2$. Thus,
$$ 0 = (Qv)^T \xi + a- \mu -2(\lambda-\mu) = (Qv)^T \xi + a+d_2 -(v_1+a+d_1) \leq \|v\| + d_2 - d_1 - v_1$$
However, by taking $d_2 > d_1 + v_1 - \|v\|$, this condition fails, which means we can discard this case.

If we consider now the family of matrices of the form \eqref{lem20-case3} that satisfy $v_2+a-d_2 \geq 0,$ we will have $\lambda= \frac{v_2+a+d_2}{2} \in \sigma_{bd}(B).$
Assuming additionally that  $\frac{v_2+a+d_2}{2}\notin \{a, d_1, d_2\}$, that is, $\lambda \notin \sigma(B)$, a similar argument to the one above yields $p_2 = \pm\|p_2\| q_2$, so we may take the diagonal matrix $D$ in the statement of the lemma to be $\diag(||p_1||, \pm||p_2||)$, which is invertible since $P$ is.
\end{proof}

In the next lemma we narrow down the possible matrices $D$ for which \cref{phidiagsin} holds and show that \cref{phidiagcon} cannot occur. 
\begin{lemma}
Let $\phi: M_3 \to M_3$ be a linear preserver of the L-spectrum with associated orthogonal matrix $Q$. Then
\begin{equation}\label{phiD}
\phi \left(\left[\begin{array}{cc}\tilde{A} & 0 \\ 0 & 0\end{array} \right]\right) = \left[ \begin{array}{cc} Q & 0 \\ 0 & 1\end{array} \right]\left[\begin{array}{cc}D\tilde{A}D^{-1} & 0 \\ 0 & 0\end{array} \right]\left[ \begin{array}{cc}  Q^T & 0 \\ 0 & 1\end{array} \right]\quad \textrm{for all $\tilde{A}\in M_2$,}
\end{equation}
where $D=I_2$ or $D=\diag(1, -1).$
\end{lemma}

\begin{proof}
We prove the result in two steps.

{\bf Step 1:} Let $A=\left[\begin{array}{ccc} 0 & 0 & 0\\ m & 0 &0 \\ v_1 & v_2 & a\end{array} \right]$, with $m \neq 0$, $a> |v_2|$, $v_1 \neq \pm v_2$, and $v_2\notin \{0, \pm a\}$. By  \cref{lem:symmetric}, $\sigma_{bd}(A) = \{\frac{\pm v_2 +a}{2} \}$ and, hence, $\{\frac{\pm v_2 +a}{2} \}\subseteq \sigma_L(\phi(A))$.  By \cref{DAD or DA^TD}, $\phi(A)$ is as in \eqref{phidiagsin} or \eqref{phidiagcon}. Assume that $\phi(A)$ is as in \eqref{phidiagcon}.
Let $D=\diag(d_1, d_2)$. Then,
$$\phi (A) = \left[ \begin{array}{cc} Q & 0 \\ 0 & 1\end{array} \right]\left[\begin{array}{ccc} 0 &\frac{d_1}{d_2} m & 0 \\0 & 0 & 0\\ v_1 & v_2 & a\end{array} \right]\left[ \begin{array}{cc}  Q^T & 0 \\ 0 & 1\end{array} \right]$$
and $\{\frac{\pm v_2 +a}{2}\} \subseteq \sigma_{bd}(\phi(A)).$ (Note that $\sigma(\phi(A))=\{0,a\}$ and $\frac{\pm v_2+a}{2} \notin \{0, a\}.$) By  \cref{lem:symmetric}, the only potential boundary L-eigenvalues of $\phi(A)$ are $\frac{\pm v_1 + a}{2}$. However, $v_1 \neq \pm v_2$, so this is a contradiction, which implies  $\phi(A)$ is as in \eqref{phidiagsin}.

\bigskip
{\bf Step 2:} Consider matrices of the form
$$B=\left[\begin{array}{cc|c} 0 & c & 0 \\ c & 0 &0 \\\hline  v_1 & v_2 & a\end{array} \right], \quad c>0.$$
Then, by the conclusion of step 1 and by  \cref{QQ^TLpres},
$$\sigma_L(B) = \sigma_L(\phi(B))= \sigma_L \left(\left[\begin{array}{cc|c} 0 & \frac{d_1}{d_2}c & 0 \\ \frac{d_2}{d_1}c & 0 &0 \\\hline v_1 & v_2 & a\end{array} \right]\right)=: \sigma_L(H).$$
Consider the family of matrices $B$ having the boundary L-eigenvalue
$$\lambda=\frac{\frac{1}{\sqrt{2}}(v_1+v_2) +a + c}{2},$$
which, by \cref{lem:symmetric}, happens if 
$$\frac{1}{\sqrt{2}}(v_1+v_2) +a -c \geq 0.$$
Since $\sigma(H)=\{a, c, -c \}$, consider those matrices $B$ for which $\lambda\notin \{a, c, -c\}$ so that $\lambda\in \sigma_{bd}(H).$
Hence, by  \cref{lem:symmetric}, 
\begin{align*}
    2\lambda &\in \begin{cases}
        \pm\frac{1}{\sqrt{d_1^2+d_2^2}}(|d_1|v_1 + |d_2|v_2) +a +|c|, \\
        \pm\frac{1}{\sqrt{d_1^2+d_2^2}}(|d_1|v_1 - |d_2|v_2) +a -|c|
    \end{cases}
\end{align*}
This implies $d_1=\pm d_2$, and after multiplying $D$ by $1/d_1$ and $D^{-1}$ by $d_1$, which does not change $D\tilde A D^{-1}$, we get $D=I_2$ or $D=\diag(1,-1).$
\end{proof}

Now we present the main result of this section, which provides a complete description of the images of matrices in $\mathcal{S}_2$ under linear preservers of the L-spectrum.
\begin{theorem}\label{thm:main3}
Let  $\phi: M_3 \to M_3$ be a linear preserver of the L-spectrum with associated orthogonal matrix  $Q$. Then
    $$\phi\left(\bmat{\tilde A & 0 \\ 0 & 0}\right) = \bmat{Q\tilde AQ^T & 0 \\ 0 & 0}\quad \text{for all $\tilde A\in M_2$.}$$
\end{theorem}
\begin{proof}
Let $A = \bmat{\tilde{A} & 0 \\ 0 & 0}.$ We know that $\phi(A)$ is as in \eqref{phiD}. Suppose $\phi\left(\bmat{\tilde A & 0 \\ 0 & 0}\right) = \bmat{QD\tilde ADQ^T & 0 \\ 0 & 0}$ for all $\tilde A$, where $D = \diag(1, -1)$. By \cref{QQ^TLpres}, we know that $\phi$ preserves the L-spectrum if and only if the map
$$\psi(A)=\bmat{Q^T & 0 \\ 0 &1} \phi(A)\bmat{Q & 0 \\ 0 & 1}$$
preserves the Lorentz spectrum, so we may suppose without loss of generality that $Q = I_2$.

Let $A=\bmat{\tilde{A} &0 \\v^T & a}$, where 
 $v=[v_1, v_2]^T\in \mathbb{R}^2$ is such that $v_1, v_2 \neq 0$ and $||v|| = 1$. Moreover, let $\tilde{A} = P \tilde{D} P^T,$ where $P = \bmat{v_1 & v_2 \\ -v_2 & v_1}$ is an orthogonal matrix and $\tilde D = \diag(\lambda, v_1v_2 + a)$ with $\lambda > v_1v_2 + a$. Note that $v_1v_2 + a\in \sigma_{int}(A)$  with corresponding L-eigenvector $[v_2/2, v_1/2, 1]^T$. Thus, $v_1v_2 + a\in \sigma_L(\phi(A))$, where  $\phi(A)= \bmat{DP\tilde DP^TD & 0 \\ v^T & a}$. 

{\bf Step 1:} Suppose $v_1v_2+a$ is a standard L-eigenvalue of $\phi(A)$. Then there is some $\xi$ with $||\xi|| \leq 1$ such that
\begin{align*}
    0 &= [H - (v_1v_2 + a)I_3]\bmat{\xi \\ 1} = \bmat{DP\tilde DP^TD\xi - (v_1v_2 + a)\xi \\ v^T\xi - v_1v_2},
\end{align*}
or equivalently,
\begin{equation}\label{thm5.57:eq1}
0=(P \tilde{D} P^T)D\xi - (v_1v_2+a)D\xi,
\end{equation}
$$0=v^T \xi - v_1v_2.$$
Since $v_1v_2 \neq 0$, the second equation implies that $\xi \neq 0$ and
$$\xi_1 = v_2 - \frac{v_2}{v_1}\xi_2.$$
As $D$ is invertible, $D\xi\neq 0$ and, from \eqref{thm5.57:eq1}, we deduce that $D\xi$ is an L-eigenvector of $P \tilde{D} P^T$ associated with $v_1v_2+a$.  This implies that $D\xi$  must be  proportional to the second column of $P$. However, this gives a contradiction as
\begin{align*}
    0 &= \det[ D\xi, \ p_2]=\det\bmat{v_2(1-\frac{\xi_2}{v_1}) & v_2 \\ - \xi_2 & v_1} = v_1v_2 \neq 0.
\end{align*}
Therefore, $v_1v_2 + a$ is not a standard L-eigenvalue of $\phi(A)$.

\bigskip

{\bf Step 2:}  Since  $v_1v_2 + a$ must be a nonstandard L-eigenvalue of $\phi(A)$, there exist $s > 0$ and $\xi$ with $||\xi|| = 1$ such that
\begin{align*}
    0 &= [H - (v_1v_2 + a)I_3]\bmat{\xi \\ 1} + s\bmat{\xi \\ -1} = \bmat{DP\tilde DP^TD\xi - (v_1v_2 + a - s)\xi \\ v^T\xi - v_1v_2 - s}.
\end{align*}
Since $\xi \neq 0$, from the first equation we see that $\xi$ is an eigenvector of $DP\tilde DP^TD$ corresponding to the eigenvalue $v_1v_2 + a - s$. Because $DP$ is orthogonal, the only eigenvalues of $DP\tilde DP^TD$ are $\lambda$ and $v_1v_2 + a$, which implies either $s = v_1v_2 + a - \lambda < 0$ or $s = 0$, a contradiction.

Thus, $v_1v_2 + a$ cannot be an L-eigenvalue of $\phi(A)$, and as a result, $\phi$ does not preserve the Lorentz spectrum. Hence, the claim follows.
\end{proof}

\subsection[Image of matrices in S3 under a linear preserver]{Image of matrices in $\mathcal{S}_3$ under a linear preserver}
We finish the proof of  \cref{main-thm} by analyzing the behavior of linear preservers  of the L-spectrum on the  subspace $\mathcal{S}_3.$ We already obtained some partial results in \cref{lem:0u-first} and \cref{[A 0; 0 0] partial result}. 

We begin by presenting an auxiliary lemma that will be used in proving some of the results in this section.
\begin{lemma}\label{lm:nB}
Let  $\phi:M_3\rightarrow M_3$ be a linear preserver of the L-spectrum with associated orthogonal matrix  $Q$.  Let
$$\phi\left(\begin{bmatrix}
0&u\\
v^T&0
\end{bmatrix}\right) = \begin{bmatrix}
\tilde{B}&Qu\\
(Qv)^T&0
\end{bmatrix}.$$  Then, for any integer $n$, we have \begin{equation}\label{eq:nmayvary}\sigma_L\left(\begin{bmatrix}
n\tilde{B}&Qu\\
(Qv)^T&0
\end{bmatrix}\right) = \sigma_L\left(\begin{bmatrix}
0&u\\
v^T&0
\end{bmatrix}\right).\end{equation}
\end{lemma}

\begin{proof} 
First we show that, for all  integers $n$,  \begin{equation}\label{eq:ntildeb}\sigma_L\left(\begin{bmatrix}n\tilde{B} & Qu \\ (Qv)^T & 0\end{bmatrix}\right)=\sigma_L\left(\begin{bmatrix}(n+1)\tilde{B} & Qu \\ (Qv)^T & 0\end{bmatrix}\right).\end{equation} 
Note that $$\begin{bmatrix}\label{matrix:ntilde(B)} n\tilde{B}&Qu\\ (Qv)^T&0 \end{bmatrix}=\begin{bmatrix} Q&0\\ 0&1\end{bmatrix} \begin{bmatrix}nQ^T\tilde{B}Q&u\\v^T&0\end{bmatrix} \begin{bmatrix}Q^T&0\\0&1\end{bmatrix},$$ 
so, by \cref{QQ^TLpres},
$$\sigma_L\left(\begin{bmatrix}n\tilde{B}&Qu\\(Qv)^T&0\end{bmatrix}\right) = \sigma_L\left(\begin{bmatrix} nQ^T\tilde{B}Q&u\\v^T&0\end{bmatrix}\right).$$ Now observe that $$\phi\left(\begin{bmatrix} nQ^T\tilde{B}Q&u\\ v^T&0\end{bmatrix}\right) = \begin{bmatrix}\tilde{B}+nQQ^T\tilde{B}QQ^T & Qu \\ (Qv)^T & 0\end{bmatrix} = \begin{bmatrix}(n+1)\tilde{B}&Qu\\(Qv)^T&0\end{bmatrix}.$$ Since $\phi$ preserves the L-spectrum, we get $$\sigma_L\left(\begin{bmatrix}n\tilde{B}&Qu\\ (Qv)^T&0 \end{bmatrix}\right)=\sigma_L\left(\begin{bmatrix}nQ^T\tilde{B}Q&u \\ v^T&0\end{bmatrix}\right) = \sigma_L\left(\begin{bmatrix} (n+1)\tilde{B}&Qu \\ (Qv)^T&0 \end{bmatrix}\right),$$
which shows \eqref{eq:ntildeb}. 

We now prove the claim by induction on $n.$ Since $\phi$ preserves the L-spectrum, we know that $$\sigma_L\left(\begin{bmatrix}0&u\\ v^T&0 \end{bmatrix}\right) = \sigma_L\left(\begin{bmatrix}\tilde{B}&Qu\\ (Qv)^T&0\end{bmatrix}\right),$$ so  \eqref{eq:nmayvary} holds for $n=1$.
Now assume that \eqref{eq:nmayvary} holds for some integer $n$. Then by \eqref{eq:ntildeb}, it holds for both $n+1$ and $n - 1$, so the claim follows for all integers.
\end{proof}

The next two lemmas analyze  the image of a basis for $\mathcal{S}_3$ under the linear preservers. Note that the assumption on the form of the images of the two matrices in the basis is a consequence of \cref{[A 0; 0 0] partial result}.
 
\begin{lemma}\label{lm:B_1&B_2}Let $\phi:M_3\rightarrow M_3$ be a linear preserver of the L-spectrum with associated orthogonal matrix $Q$, and
let $$\phi\left( \begin{bmatrix}0 & Q^Te_1 \\ 0 & 0\end{bmatrix}\right) = \begin{bmatrix} \tilde{B}_1 & e_1 \\ 0 & 0\end{bmatrix}\quad \textrm{ and}\quad \phi\left( \begin{bmatrix} 0 & Q^Te_2 \\ 0 & 0\end{bmatrix}\right) = \begin{bmatrix} \tilde{B}_2 & e_2 \\ 0 & 0\end{bmatrix}.$$  Then there exist $x, y\in \mathbb{R}$ such that 
$$
\tilde{B}_1 = \begin{bmatrix} 0 & x \\ 0 & y\end{bmatrix} \quad \textrm{and} \quad \tilde{B}_2 = \begin{bmatrix}-x & 0 \\ -y & 0 \end{bmatrix}.
$$
\end{lemma}

\begin{proof}
Let $u=\bmat{u_1, u_2}^T\in\mathbb{R}^2$ be a nonzero vector, and let $A = \begin{bmatrix} 0 & Q^Tu \\ (Q^Tu)^T & 0\end{bmatrix}.$ Then by \cref{thm:main2}
and by the linearity of $\phi$,
$$
\phi(A) = \begin{bmatrix} u_1 \tilde{B}_1 + u_2\tilde{B}_2 & u \\ u^T & 0\end{bmatrix}.
$$

By  \cref{symmetricua}, $\|u\| \in \sigma_{bd}(A) \subseteq \sigma_L(\phi(A)).$ Suppose that $\|u\| \in \sigma_{int}(\phi(A)).$ Then there exists $\xi$ with $\|\xi\|<1$ such that 
$$
0 = \begin{bmatrix} u_1 \tilde{B}_1 + u_2\tilde{B}_2 - \|u\|I_2 & u \\ u^T & -\|u\| \end{bmatrix} \begin{bmatrix} \xi \\ 1\end{bmatrix}.
$$
Then from the second equation, we get $$\|u\| = u^T\xi \leq \|u\|\|\xi\| < \|u\|,$$ a contradiction. Therefore, $\|u\| \in \sigma_{bd}(\phi(A)).$ This means that there exist $ s \geq 0$,  $\mu$, and $\xi$ such that $\|u\|=\mu+s,$ $\|\xi\|=1,$ and
$$
0 = \begin{bmatrix} u_1\tilde{B}_1 + u_2\tilde{B}_2 - \mu I_2 & u \\ u^T & -\|u\| - s\end{bmatrix}\begin{bmatrix} \xi \\ 1\end{bmatrix},
$$
from which it follows that $s=u^T\xi - \|u\| \leq 0.$ This implies $s=0$ and hence $\mu=\|u\|.$ 

The only solution $\xi$ of the equation $u^T\xi-\|u\|=0$  on the unit circle is $\xi=\frac{u}{\|u\|}.$ Thus,  we have
\begin{align}
0 &= (u_1 \tilde{B}_1+u_2 \tilde{B}_2 - \|u\| I_2)\xi + u 
= (u_1\tilde{B}_1 + u_2\tilde{B}_2) \frac{u}{\|u\|} \label{equation:tilde(B)1}
\end{align}
for all $u \neq 0.$ Choosing $u_1 \neq u_2=0$ gives $\tilde{B}_1e_1=0$, and choosing $u_2 \neq u_1 = 0$ gives $\tilde{B}_2e_2=0.$ Plugging  these results in \eqref{equation:tilde(B)1} gives 
$$0 = u_1u_2 (\tilde{B}_1 e_2 + \tilde{B}_2 e_1).$$
Since this must hold when $u_1, u_2 \neq 0$, we deduce
$\tilde{B}_1e_2 = -\tilde{B}_2e_1,$ and the result follows.
\end{proof}

\begin{lemma}\label{x=0 iff y=0}
Let $\tilde{B}_1 = \begin{bmatrix}0&x\\ 0&y\end{bmatrix}$ and $ \tilde{B}_2 = \begin{bmatrix}-x&0\\ -y&0\end{bmatrix}$ be as in  \cref{lm:B_1&B_2}. Then $x = 0$ if and only if $y = 0$.
\end{lemma}

\begin{proof}
Suppose that $y=0$ and $x \neq 0$, and let
$$H_n := \bmat{n\tilde B_1 & e_1 \\ e_1^T & 0}$$
for any integer $n$. Since 
$$\phi\left(\begin{bmatrix}0&Q^Te_1\\ (Q^Te_1)^T&0\end{bmatrix}\right) = \begin{bmatrix} \tilde{B}_1&e_1\\ e_1^T&0\end{bmatrix} = H_1,$$  \cref{lm:nB} implies that
\begin{equation}\label{eq:nb1spectrum}\sigma_L(H_n) = \sigma_L(H_1) = \sigma_L\left(\begin{bmatrix}0&Q^Te_1\\ (Q^Te_1)^T&0\end{bmatrix}\right) = \sigma_L\left(\begin{bmatrix}0&e_1\\ e_1^T&0\end{bmatrix}\right) = \{\pm1\},\end{equation} for all integers $n$, where the last equality follows from \cref{symmetricua}. Next we observe that $ny=0$ is a standard eigenvalue of $H_n$  with associated eigenvector $\xi=\begin{bmatrix}0 & -\frac{1}{nx} & 1\end{bmatrix}^T$. For $n$ sufficiently large, we get $\left|\frac{1}{nx}\right| \leq 1$ which implies that $\xi$ lies in the Lorentz cone and hence $0\in\sigma_L(H_n)$. However, this is a contradiction by  \eqref{eq:nb1spectrum}. Thus,  if  $y=0$, then $x=0$.

By applying a similar argument to $\begin{bmatrix}
n\tilde{B}_2&e_2\\
e_2^T&0
\end{bmatrix}$, we may likewise conclude that if $x=0$, then $y=0$.
\end{proof}

We now arrive at the main result in this section and the final piece for the proof of  \cref{main-thm}.

\begin{theorem}\label{thm:main4}
Let $\phi:M_3\rightarrow M_3$ be a linear preserver of the L-spectrum with associated orthogonal matrix $Q$. Then for all $u\in \mathbb{R}^{2}$,
$$\phi \left (\begin{bmatrix} 0 & u\\0. & 0 \end{bmatrix}\right)=\begin{bmatrix} 0 & Qu \\ 0 & 0 \end{bmatrix}.$$
\end{theorem}

\begin{proof}
Note that it is enough to prove
 $$\phi\left( \begin{bmatrix}0 & Q^Te_1 \\ 0 & 0\end{bmatrix}\right) = \begin{bmatrix} 0 & e_1 \\ 0 & 0\end{bmatrix}\quad \textrm{ and}\quad \phi\left( \begin{bmatrix} 0 & Q^Te_2 \\ 0 & 0\end{bmatrix}\right) = \begin{bmatrix} 0 & e_2 \\ 0 & 0\end{bmatrix}$$
 since $\{Q^Te_1, Q^Te_2\}$ is a basis for $\mathbb{R}^2$. That is, we want to show that $\tilde B_1 = \tilde B_2 = 0$, where the matrices $\tilde B_1 = \bmat{0 & x \\ 0 & y}$ and $\tilde B_2 = \bmat{-x & 0 \\ -y & 0}$ are as in \cref{lm:B_1&B_2}.
 
Suppose for sake of contradiction that $x\neq0$, which implies $y \neq 0$ by  \cref{x=0 iff y=0}. Let
$$H_n := \bmat{n\tilde B_1 & e_1 \\ e_1^T & 0}$$
for any integer $n$. By \eqref{eq:nb1spectrum}, we have $\sigma_L(H_n) = \{\pm1\}$ for all $n$.

Let $n\neq0$, $\mu=ny$, and $\lambda = \mu + s$ for some real number $s$. Then $\lambda \in \sigma_{bd}(H_n)$ if and only if there exists some $\xi=[\xi_1,\xi_2]^T \in \mathbb{R}^2$ such that
\begin{equation}\label{eq:eigenvaluesb1}\begin{bmatrix}-ny&nx&1\\ 0&0&0\\ 1&0&-ny-2s\end{bmatrix} \begin{bmatrix}\xi_1\\\xi_2\\1\end{bmatrix}=0, \quad\|\xi\|=1, \quad \text{and} \quad s\geq0,\end{equation} or equivalently,
\begin{equation}\label{condition8}
    -ny\xi_1+nx\xi_2+1 = 0,
\end{equation}
\begin{equation}\label{condition9}
    \xi_1-ny-2s = 0.
\end{equation}

We now show that there exists some real number $N$ such that either $N$ is positive and it is possible to satisfy \eqref{eq:eigenvaluesb1} for any integer $n>N$, or $N$ is negative and it is possible to satisfy \eqref{eq:eigenvaluesb1} for any integer $n<N$. Let $$N:=\begin{cases}\min\left(-\frac{1}{\sqrt{x^2+y^2}},-\frac{3}{y}\right) & \text{if }y>0; \\\max\left(\frac{1}{\sqrt{x^2+y^2}},-\frac{3}{y}\right) & \text{if }y<0.\end{cases}$$
For each $n\neq 0$, \eqref{condition8} has a solution $\xi^{(n)}$ with $||\xi^{(n)}|| = 1$ if and only if $|n| \geq \frac{1}{\sqrt{x^2+y^2} }$, which holds whenever $|n|>|N|$.

Note that \eqref{condition9} is equivalent to $$s = \frac{\xi_1-ny}{2}.$$ For each $\xi^{(n)}$ with $||\xi^{(n)}|| = 1$ satisfying \eqref{condition8} and for each $n$ in the range specified above, we know that
$$\xi^{(n)}_1\geq -1 > -3 \geq Ny > ny,$$
which guarantees that
$$s_n := \frac{\xi^{(n)}_{1}-ny}{2}>0.$$
To conclude the argument, choose $n$ in the range specified above so that $\xi^{(n)}$ and $s_n$ satisfy (\ref{eq:eigenvaluesb1}). Then $\lambda = \mu + s_n \in \sigma_{bd}(H_n) = \{\pm 1\}$. However, this gives a contradiction since
$$
    \lambda = \frac{\xi^{(n)}_{1}+ny}{2} < \frac{\xi_{1}^{(n)}-Ny}{2} \leq \frac{1 - 3}{2} = -1.$$
It follows that  $x= 0$ and hence $y = 0$ by \cref{x=0 iff y=0}. Therefore, $\tilde{B}_1=\tilde{B}_2=0$.
\end{proof}

\section{Conclusions}
In this paper, we have analyzed the linear preservers of the Lorentz spectrum of $3\times 3$ real matrices and proven that every such linear map $\phi$ must be of the form
$$\phi(A) = \bmat{Q &0 \\ 0 &1} A \bmat{Q^T & 0 \\ 0 & 1}$$
for some orthogonal $Q\in M_2$, as conjectured in \cite{maribel1}. An immediate corollary of this result is that the linear preservers of the L-spectrum on $M_3$ must take interior (resp.\ boundary) L-eigenvalues to interior (resp.\ boundary) L-eigenvalues. Our proof relies on the particular form of $3 \times 3$ matrices with infinitely many L-eigenvalues, which makes it difficult to generalize this result to higher dimensions since $n \times n$ matrices with this property can be much more complicated. Thus, it is likely that different techniques will be necessary to prove the corresponding result for $M_n$, but we hope that some of the strategies we developed in this paper will still be applicable in this case.

\appendix
\section{L-spectrum of some special matrices}

Here we provide some results for $n \times n$ matrices, where $n \geq 3$, and one result specific to $3\times 3$ matrices.

\subsection[Results for n x n matrices]{Results for $n\times n$ matrices}

The following result is an immediate consequence of Corollary 3.3 in \cite{maribel1}.
\begin{lemma}\label{diagonalca}
Let $A=\begin{bmatrix} cI_{n - 1} & 0 \\ 0 & a \end{bmatrix} \in M_n$. Then
$$\sigma_L(A) = \left\{ \begin{array}{cc} \{a\} & \textrm{if $c>a$} \\ \left\{ a, \frac{a+c}{2} \right\} & \textrm{if $c\leq a$}\end{array} \right..$$
\end{lemma}

The next result follows from Theorem 3.4 in \cite{maribel1}.

\begin{lemma}\label{symmetricua}
 Let
\[
A=\left[
\begin{array}
[c]{cc}
0 & u\\
v^{T} & a
\end{array}
\right]  ,
\]
where $u,v\in\mathbb{R}^{n-1}$ are not both zero and $a\in\mathbb{R}$. Then

\begin{enumerate}
\item $0\in \sigma_{int}(A)$ (resp.\ $0 \in \sigma_{bd}(A)$)  if and only if $u=0$ and
$|a|<\Vert v\Vert$ (resp.\ $|a|\leq \|v\|$).

\item If $\lambda\neq 0,$ then $\lambda \in \sigma_{int}(A)$ (resp.\ $\lambda \in \sigma_{bd}(A)$)  if and only if $|\lambda|>\Vert u\Vert$ (resp.\ $|\lambda| \geq \|u\|$) and $\lambda^{2}-a\lambda-v^{T}u=0.$

\item If $u\neq0$, then $\lambda$ is a nonstandard Lorentz eigenvalue of $A$
if and only if one of the following holds:

\begin{enumerate}
\item[(i)] $v^{T}u+a\Vert u\Vert-\Vert u\Vert^{2}>0$ and $\lambda=\frac
{1}{2\Vert u\Vert}[{a \Vert u\Vert+{\Vert u\Vert}^{2}+v^{T}u}]$.

\item[(ii)] $\Vert u\Vert^{2}+a\Vert u\Vert-v^{T}u>0$ and $\lambda=\frac
{1}{2\Vert u\Vert}[{a \Vert u\Vert-{\Vert u\Vert}^{2}-v^{T}u}]$.
\end{enumerate}

\item If $u=0$ (and hence $v\neq0$), then $\lambda$ is a nonstandard Lorentz eigenvalue
of $A$ if and only if

\[
\lambda\in\left[  \frac{a-\Vert v\Vert}{2},\frac{a+\Vert v\Vert}{2}\right]
\cap(0,\infty).
\]

\end{enumerate}

\end{lemma}

\begin{lemma}\label{spec-inf}
Let $$A = \left[ \begin{array}{cc} c I_{n-1} & 0 \\ v^T & a \end{array} \right], \quad \text{where} \quad a,c\in \mathbb{R} \quad \text{and} \quad 0\neq v\in \mathbb{R}^{n-1}.$$
Then
\begin{enumerate}
\item if $c <a -\|v\| $, then $\sigma_{int}(A)=\{a\}$ and $$\sigma_{bd}(A) =  \left[  \frac{c+a - \|v\|}{2} , \frac{c+a + \|v\|}{2} \right] .$$
Moreover, $\sigma_{int}(A) \cap \sigma_{bd}(A) = \emptyset$. 
\item if $c=a-\|v\|$, then $\sigma_{int}(A)=\{a\}= \{\|v\|+c\},$ and
$$\sigma_{bd}(A) =[c, \|v\|+c].$$ 
Therefore, $a\in \sigma_{int}(A)\cap \sigma_{bd}(A).$
\item if $ a- \|v\|< c <a+ \|v\|$, then $\sigma_{int}(A)=\{a, c\}$ and
$$\sigma_{bd}(A) =  \left[ c, \frac{c+a + \|v\|}{2} \right] .$$
Therefore, $c \in \sigma_{int}(A)\cap \sigma_{bd}(A).$ Moreover,   $a\in \sigma_{int}(A)\cap \sigma_{bd}(A)$ if and only if  $c \leq a$.
\item if $c= a+\|v\| $, then $$\sigma_{int}(A) =\{a\}\quad \textrm{ and} \quad \sigma_{bd}(A) =\{c\}.$$
\item if $a+\|v\| < c$, then
$$\sigma_{int}(A)=\{a\}\quad \textrm{ and}\quad \sigma_{bd}(A)=\emptyset.$$
\end{enumerate}
\end{lemma}

\begin{proof}
This result follows from \cref{symmetricua} and the fact that $\sigma_L(A+\gamma I) = \sigma_L(A) + \gamma$ for all $\gamma\in \mathbb{R}.$
\end{proof}

 \begin{lemma}\label{lem:a-int}
Let $A = \bmat{\tilde A & u \\ 0 & a}$, where $\tilde A \in M_{n - 1}$ and $u \in \R^{n - 1}$ are fixed. Then $a \in \sigma_{int}(A)$ for all sufficiently large $a$.
\end{lemma}
\begin{proof}
$a$ is an interior L-eigenvalue of $A$ if and only if there exists $\xi$ with $||\xi|| < 1$ such that
$$(\tilde A - aI)\xi + u = 0.$$
For $a > \rho(\tilde A)$, the matrix $\tilde A - aI$ is invertible, so we have
$$\xi = -(\tilde A - aI)^{-1}u.$$
Then since $\det(\tilde A - aI)$ is a polynomial in $a$ of degree $n - 1 \geq 2$,
\begin{align*}
    ||\xi||^2& = ||(\tilde A - aI)^{-1}u||^2 = \left\|\frac{(\adj(\tilde A) - aI)u}{ \det(\tilde A-a I)}\right\|^2 
    = \frac{||\adj(\tilde A)u||^2 - 2au^T\adj(\tilde A)u + a^2||u||^2}{\det(\tilde A-aI)^2}
\end{align*}
approaches zero as $a \to \infty$. Thus, we likewise have $||\xi|| < 1$ for all sufficiently large $a$.
\end{proof}

\subsection[A result for 3 x 3 matrices]{A result for $3\times 3$ matrices}

\begin{lemma}\label{lem:symmetric}
Let 
$$A=\left[\begin{array}{ccc} 0 & c  & 0\\ d & 0 & 0 \\ v_1 & v_2 & a\end{array} \right], \quad \text{where} \quad cd \geq 0.$$
Then
$$\sigma_{bd}(A)= \left\{\begin{array}{cc} \frac{1}{2} \left(\pm \sqrt{\frac{c}{c+d}}v_1 \pm \sqrt{\frac{d}{c+d}} v_2 +a +\sqrt{cd}  \right) & \textrm{if $\pm\sqrt{\frac{c}{c+d}}v_1\pm \sqrt{\frac{d}{c+d}} v_2 + a - \sqrt{cd} \geq 0$}\\
\frac{1}{2} \left(\pm \sqrt{\frac{c}{c+d}}v_1 \mp \sqrt{\frac{d}{c+d}} v_2 +a -\sqrt{cd}   \right) & \textrm{if $\pm\sqrt{\frac{c}{c+d}}v_1\mp \sqrt{\frac{d}{c+d}} v_2 + a + \sqrt{cd} \geq 0$}
\end{array}\right..
$$
\end{lemma}

\begin{proof}
Let $\lambda \in \sigma_{bd}(A)$. Then $\lambda=\mu+s$ with $s\geq 0$, and there exists $\xi=[\xi_1, \xi_2]^T$ with $\|\xi\|=1$ such that
$$-\mu \xi_1 + c\xi_2=0$$
$$d\xi_1 -\mu \xi_2 =0$$
$$v_1 \xi_1+v_2\xi_2 +a - \mu-2s =0.$$

Assume that $\mu=0$. Then from the first equation, we get $\xi_2=0$ since $c \neq 0$. From the second equation, we get $\xi_1$ since $d\neq 0$, a contradiction as $\xi \neq 0$. Thus, $\mu \neq 0$, and the two first equations yield 
$$(cd -\mu^2)\xi_2=0.$$
Since $\xi_2=0$ would imply $\xi_1=0$, contradicting $\|\xi\|=1$, we have $\mu = \pm \sqrt{cd}$ and $\xi_1= \frac{c}{\pm \sqrt{cd}} \xi_2= \pm \sqrt{\frac{c}{d}} \xi_2$.  Thus, since $\|\xi\|=1$, we have
$$1=\xi_1^2 + \xi_2^2 = \frac{c+d}{d} \xi_2^2.$$
Hence, if $\mu= \sqrt{cd}$, then
$$\xi_1= \pm \sqrt{\frac{c}{c+d}} \quad \text{and} \quad \xi_2= \pm \sqrt{\frac{d}{c+d}},$$
and if $\mu = -\sqrt{cd}$, then
$$\xi_1= \pm \sqrt{\frac{c}{c+d}}\quad \text{and} \quad \xi_2= \mp \sqrt{\frac{d}{c+d}}.$$

From the third equation, we get
$$ v_1\xi_1 + v_2 \xi_2+a \mp \sqrt{cd} - 2s =0,$$
or equivalently,
$$s = \frac{(v_1\xi_1+ v_2\xi_2) + a \mp \sqrt{cd}}{2},$$
which yields the claimed boundary L-eigenvalues $\lambda = \mu + s$ with the condition $s \geq 0$.
\end{proof}

\end{document}